\documentclass[11pt,reqno,twoside]{amsart}
\usepackage{graphicx}
\usepackage{amssymb}
\usepackage{epstopdf}
\usepackage[asymmetric,top=3.5cm,bottom=3.5cm,left=3.1cm,right=3.1cm]{geometry}
\usepackage{booktabs} 
\usepackage{array} 
\usepackage{paralist} 
\usepackage{verbatim} 
\usepackage{subfig} 
\usepackage{tabularx}
\usepackage{amsmath,amsfonts,amsthm,mathrsfs,amssymb,cite}
\usepackage[usenames]{color}
\usepackage{bm}

\newtheorem{thm}{Theorem}[section]

\newtheorem{lem}{Lemma}[section]

\theoremstyle{definition}

\theoremstyle{remark}

\newtheorem{rem}{Remark}[section]
\numberwithin{equation}{section}

\title[Plasmonic resonance and cloaking beyond the quasistatic limit]{On anomalous localized resonance and plasmonic cloaking beyond the quasistatic limit}

\author{Hongjie Li}
\address{Department of Mathematics, Hong Kong Baptist University, Kowloon Tong, Hong Kong SAR.}
\email{hongjie$_{-}$li@yeah.net}

\author{Hongyu Liu}
\address{Department of Mathematics, Hong Kong Baptist University, Kowloon Tong, Hong Kong SAR.}
\email{hongyu.liuip@gmail.com; hongyuliu@hkbu.edu.hk}

\begin{document}
\maketitle

\begin{abstract}

In this paper, we give the mathematical construction of novel core-shell plasmonic structures that can induce anomalous localized resonance and invisibility cloaking at certain finite frequencies beyond the quasistatic limit. The crucial ingredient in our study is that the plasmon constant and the loss parameter are constructed in a delicate way that are correlated and depend on the source and the size of the plasmonic structure. As a significant byproduct of this study, we also derive the complete spectrum of the Neumann-Poinc\'are operator associated to the Helmholtz equation with finite frequencies in the radial geometry. The spectral result is the first one in its type and is of significant mathematical interest for its own sake.

\medskip

\medskip

\noindent{\bf Keywords:}~~anomalous localized resonance, plasmonic material,  core-shell structure, beyond quasistatic limit, Neumann-Poinc\'are operator, spectral


\end{abstract}

\section{Introduction}

\subsection{Background and motivation}

Recently, there is significant interest on the mathematical study of plasmon materials in understanding their peculiar and distinctive behaviours and the potential striking applications. Plasmon materials are a type of metamaterial that are artificially engineered to allow the presence of negative material parameters, including negative permittivity and permeability in electromagnetism \cite{Ack13,Ack14,ARYZ,Bos10,Brl07,CKKL,Klsap,LLL,GWM1,GWM2,GWM3,GWM4,GWM6,GWM7,GWM8,GWM9,Pen1,Pen2,Ves}, negative density and refractive index in acoustics \cite{ADM,AMRZ,AKL,KLO}, and negative Lam\'e parameters in linear elasticity \cite{AKKY,AKKY2,AKM1,AKM2,DLL,LiLiu2d,LiLiu3d,LLL2,KM,LLBW}.

Among the various plasmonic phenomena, we are particularly interested in the anomalous localized resonance (ALR) and its associated cloaking effect that were first discovered by Milton and Nicorovici in \cite{GWM3}. Mathematically, the plasmon resonance is associated to the infinite dimensional kernel of a certain non-elliptic partial differential operator (PDO). In fact, the presence of negative material parameters breaks the ellipticity of the underlying partial differential equations (PDEs) that govern the various physical phenomena. Consequently, the non-elliptic PDO may possess a nontrivial kernel, which in turn may induce various resonance phenomena due to appropriate external excitations. The anomalous localized resonance is particularly delicate and intriguing. It demonstrates highly oscillating behaviour that is manifested by the energy blowup. Moreover, the resonance is localized in the sense that the resonant field is confined to a bounded region with a sharp boundary not defined by the discontinuity of the material parameters, and outside that region, the resonant field converges to a smooth one. The resonance strongly depends not only on the form of the external source, but also on the location of the source. If ALR occurs, one can show that for a certain external excitation, both the source and the plasmonic structure are invisible to the external field observation, namely cloaking is achieved (cf. \cite{Ack13,GWM3}). This is referred to as cloaking due to anomalous localized resonance (CALR) in the literature. Physically speaking, if CALR occurs, small objects locating beside the plasmonic structure are also invisible to the external field observation and this was confirmed in \cite{Ngu1}. Clearly, the occurrence of ALR depends on the delicate structure of the plasmon device, in particular, the appropriate choice of the plasmonic parameters. It turns out that the choice of the plasmon parameters that can induce the ALR and cloaking is closely connected to the spectrum of the classical Neumann-Poinc\'are operator in potential theory. All of those distinctive features make the ALR a unique subject for mathematical study.

Plasmon resonance and anomalous localized resonance have been extensively investigated and we refer to the aforementioned literature for related existing studies. In this paper, we consider the ALR and cloaking described by the 2D and 3D Helmholtz systems. In two dimensions, it describes the transverse electromagnetic wave propagation in the time-harmonic regime, and in three dimensions, it describes the time-harmonic acoustic wave propagation. We are mainly concerned with the mathematical study and provide a uniform treatment of the ALR and plasmonic cloaking associated to the Helmholtz system in both 2D and 3D. One of the major contributions is the mathematical construction of novel core-shell plasmonic structures that can induce ALR and cloaking at certain finite frequencies beyond the quasistatic limit. In many of the existing studies on plasmon resonance and ALR, the quasistatic approximation has played a critical role. There are also several studies that go beyond the quasistatic limit \cite{GWM3,KLO,Ngu2}. In \cite{GWM3}, double negative materials are employed in the shell and in \cite{Ngu2}, in addition to the employment of double negative materials, a so-called double-complementary medium structure is incorporated into the construction of the plasmonic device. CALR is shown to occur for the constructions in \cite{GWM3,Ngu2}. In \cite{KLO}, it is actually shown that resonance does not occur for the classical core-shell plasmonic structure without the quasistatic approximation as long as the core and shell are strictly convex. The result in \cite{GWM3,KLO} indicates that if one intends to construct a core-shell plasmonic structure which can induce CALR, the plasmonic configuration has to be properly chosen. In the current study, we do not employ the double negative materials in the shell. The crucial point of our study is that we take all the ingredients in a plasmonic configuration including the various material parameters, the size parameters, the external source as well as the frequency as a whole system. By delicately balancing all those configuration ingredients, we show that CALR can still be achieved in certain scenarios.

It is worth pointing out that as a significant byproduct of this study, we also derive the complete spectrum of the Neumann-Poinc\'are operator associated to the Helmholtz equation with finite frequencies in the radial geometry. The spectral result is also the first one in its type and of significant mathematical interest for its own sake. Finally, we would like to emphasize that the current article is a piece of theoretical work, and the practical construction or fabrication of the mathematically predicted new structures is beyond the scope of our study.

\subsection{Mathematical setup}

\begin{figure}[t]
  \centering
 {\includegraphics[width=4cm]{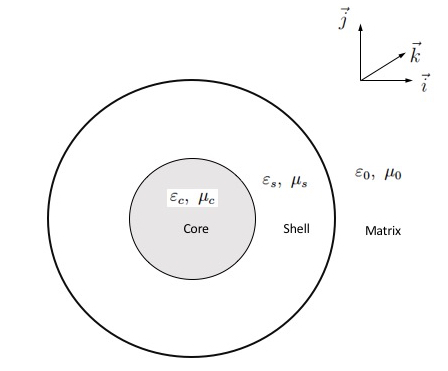}}
  \caption{Schematic illustration of the core-shell-matrix structure of the plasmonic configuration in transverse electromagnetic wave propagation. }
  \label{fig:1}
\end{figure}

We first consider the transverse electromagnetic wave propagation corresponding to an infinitely long cylindrical structure. Henceforth, we let $B_r$, $r\in\mathbb{R}_+$, signify a 3D ball or a 2D disk of radius $r$ and centred at the origin. Consider a material cylinder that is infinite along the $\vec{k}$-direction and of a cross section being the concentric disks $B_{r_i}$ and $B_{r_e}$, $r_i<r_e$; see Fig.~\ref{fig:1} for a schematic illustration. Suppose the electric permittivity and magnetic permeability in the core $B_{r_i}$ and in the shell $B_{r_e}\backslash\overline{B_{r_i}}$ are, respectively, given by $(\varepsilon_c, \mu_c)$ and $(\varepsilon_s,\mu_s)$. The exterior of the cylinder is the matrix which is supposed to be uniformly homogeneous. The permittivity and permeability of the matrix are, respectively, given by two positive constants $\varepsilon_0$ and $\mu_0$. Furthermore, we assume that the material in the shell is lossy, and its electric conductivity is signified by $\sigma_s$. An external electric/magnetic source of the form $f(x_1,x_2)\vec{k}$ is assumed, where for $x\in\mathbb{R}^3$, we use the convention $x=x_1\vec{i}+x_2\vec{j}+x_3\vec{k}$. $f(x_1,x_2)$ is compactly supported outside $B_{r_e}$. Both transverse electric (TE) and transverse magnetic (TM) poloarizations can be considered, and $e^{-\mathrm{i}\omega t}$ time-harmonic convention is also assumed.

The transverse electromagnetic scattering corresponding to the configuration described above is governed by the following 2D Helmholtz system
\begin{equation}\label{eq:2DH}
\begin{cases}
\displaystyle{\nabla\cdot\left[\frac{1}{\alpha(x_1,x_2)}\nabla u(x_1,x_2)\right]+\omega^2\beta(x_1,x_2) u(x_1,x_2)=f(x_1,x_2),}\medskip\\
\displaystyle{\lim_{|x'|\rightarrow+\infty} |x'|^{1/2}\left(\frac{x'}{|x'|}\cdot\nabla u-\mathrm{i} k u  \right)=0,\quad k=\omega\sqrt{\varepsilon_0\mu_0},}
\end{cases}
\end{equation}
where $x'=(x_1,x_2)$. In equation \eqref{eq:2DH}, for the TM wave scattering, we have $u=\mathbf{E}\cdot \vec{k}$ with $\mathbf{E}$ denoting the electric wave field, and
\begin{equation}\label{eq:tm2d}
(\alpha,\beta)=(\varepsilon_c,\mu_c)\chi(B_{r_i})+(\varepsilon_s+\mathrm{i}\delta,\mu_s)\chi(B_{r_e}\backslash\overline{B_{r_i}})+(\varepsilon_0,\mu_0)\chi(\mathbb{R}^2\backslash\overline{B_{r_e}}),
\end{equation}
where $\delta=\sigma_s/\omega$ and $\chi$ signifies the characteristic function; whereas for the TE wave scattering, we have $u=\mathbf{H}\cdot \vec{k}$ with $\mathbf{H}$ denoting the magnetic wave field, and
\begin{equation}\label{eq:tm2d}
(\alpha,\beta)=(\mu_c,\varepsilon_c)\chi(B_{r_i})+(\mu_s,\varepsilon_s+\mathrm{i}\delta)\chi(B_{r_e}\backslash\overline{B_{r_i}})+(\mu_0, \varepsilon_0)\chi(\mathbb{R}^2\backslash\overline{B_{r_e}}).
\end{equation}

In the three-dimensional case, the Helmholtz system of the form \eqref{eq:2DH} can be used to describe the acoustic wave scattering. Let $\rho(x)$,$n(x)$ and $\tau(x)$, respectively, denote the density, refractive index and absorption coefficient of an acoustic medium. Consider an acoustic configuration of the following form,
\begin{equation}\label{eq:a1}
(\alpha, \beta)=(\rho_c,n_c^2)\chi(B_{r_i})+(\rho_s, n_s^2+\mathrm{i}\delta)\chi(B_{r_e}\backslash\overline{B_{r_i}})+(\rho_0, n_0^2)\chi(\mathbb{R}^2\backslash\overline{B_{r_e}}),
\end{equation}
where $\delta=\tau/\omega$ and $\rho_0, n_0$ are two positive constants, signifying the uniformly homogeneous matrix. Let $f(x)$ denote an acoustic source compactly supported in the matrix. Then the acoustic scattering is described by the Helmholtz system \eqref{eq:2DH} with $x'$ replaced by $x$, $(\alpha, \beta)$ replaced by \eqref{eq:a1}, $k=\omega n_0\sqrt{\rho_0}$ and $|x'|^{1/2}$ by $|x|$.

In order to provide a uniform treatment of the transverse electromagnetic scattering and the acoustic scattering, we introduce the following Helmholtz system in $\mathbb{R}^N$, $N=2, 3$,
\begin{equation}\label{eq:helm1}
\begin{cases}
&\nabla\cdot(\epsilon(x)\nabla u(x))+k^2 q(x) u(x)=f(x)\quad x\in\mathbb{R}^N,\medskip\\
&\displaystyle{\lim_{|x|\rightarrow+\infty} |x|^{(N-1)/2}\left(\frac{x}{|x|}\cdot\nabla u-\mathrm{i} k u  \right)=0}.
\end{cases}
\end{equation}
In \eqref{eq:helm1}, the material configuration is specified as follows,
\begin{equation}\label{eq:mc1}
(\epsilon,q)=(\epsilon_c,1)\chi(B_{r_i})+(\epsilon_s+\mathrm{i}\delta, 1)\chi(B_{r_e}\backslash\overline{B_{r_i}})+(1, 1)\chi(\mathbb{R}^N\backslash\overline{B_{r_e}}).
\end{equation}

Several remarks are in order as follows. The Helmholtz system \eqref{eq:helm1}--\eqref{eq:mc1} includes \eqref{eq:2DH}-\eqref{eq:a1} as particular cases if the corresponding material parameters are appropriately chosen. For example, for the TM wave scattering, one can choose $\mu_c=\mu_s=\mu_0$. Then by a standard normalisation, one can reduce \eqref{eq:2DH}--\eqref{eq:tm2d} to \eqref{eq:helm1}--\eqref{eq:mc1}. It is emphasized that in our study, the derivation of the plasmonic cloaking device is a constructive procedure. That is, with the appropriate design of the material configuration for the core-shell-matrix structure, the ALR and cloaking are achievable. By introducing \eqref{eq:mc1}, the constructive procedure is reduced to properly choosing the parameters $\epsilon_c, \epsilon_s$ and $\delta$. Second, the lossy parameter $\delta$ is attached to $\epsilon$ in the shell, and instead, it can also be attached to $q$ in the shell. Indeed, following a similar reduction procedure as above, one can show that for the TE and acoustic scattering, the lossy parameter is actually attached to $q$ in the shell. However, in order to simplify the exposition and provide a uniform mathematical treatment, we only consider the case that $\delta$ is attached to $\epsilon$ in the shell. Nevertheless, we emphasize that all of our subsequent mathematical arguments can be readily extended to the case that $\delta$ is instead attached to $q$ in the shell with slight modifications. Indeed, theoretically speaking, $\delta$ is a regularisation parameter and for most of the cases in our study, the resonance occurs in the limits as $\delta\rightarrow+0$. Moreover, for many of the existing studies in the literature, in order to conveniently treat the quasistatic approximation with $k$ either formally taking to be zero or $k\ll 1$, one always attaches the regularisation parameter $\delta$ to $\epsilon$. We also adopt this convention such that our current study can be easily connected to the existing ones in the literature.

For the solution $u\in H_{loc}^1(\mathbb{R}^N)$ to \eqref{eq:helm1}--\eqref{eq:mc1}, we define
\begin{equation}\label{eq:energya1}
\mathscr{E}_\delta[u]:=\delta \int_{B_{r_e}\backslash\overline{B_{r_i}}} \left|\nabla u(x) \right|^2\ dx,
\end{equation}
which signifies the energy dissipation of the Helmholtz system. The configuration $(\epsilon_c,\epsilon_s+\mathrm i\delta, f)$ in \eqref{eq:helm1}--\eqref{eq:mc1} is said to be {\it resonant} if
\begin{equation}\label{eq:res1a2}
\limsup_{\delta\rightarrow\delta_0} \mathscr{E}_\delta[u]=+\infty,
\end{equation}
where $\delta_0$ is a fixed positive constant. It is remarked that for the existing studies in the literature, it is always assumed that $\delta_0=0$. However, without the quasistatic approximation in the current article, one sometime would need a significant loss in the plasmonic shell in order to induce resonances. This is in sharp difference from the ALR and cloaking in the quasistatic regime. An explanation is that the Neumann-Poinc\'re operator associated to the Helmholtz system \eqref{eq:helm1} at finite frequencies is non self-adjoint, and the eigenvalues are complex; see Remark~\ref{rem:lll1} for more relevant discussion. Nevertheless, under generic conditions of the source term, we can still retain that $\delta_0=0$ in most of the cases in our subsequent study. If in addition to \eqref{eq:res1a2}, there exists $R'>r_e$ and $C\in\mathbb{R}_+$ such that the following condition holds,
\begin{equation}\label{eq:res2}
|u(x)|\leq C\quad\mbox{for}\ \ |x|\geq R',
\end{equation}
then we say that anomalous localized resonance (ALR) occurs. If ALR occurs, then by straightforward scaling analysis one can show that the energy dissipation, associated to $(\epsilon,q,f/\eta)$ with $\eta:=\sqrt{\mathscr{E}_\delta}$, is normalized and the wave field outside $B_{R'}$ tends to zero as $\delta\rightarrow\delta_0$. Hence, both the plasmonic structure $(B_{r_i}; \epsilon_c)\oplus (B_{r_e}\backslash\overline{B_{r_i}}; \epsilon_s+\mathrm{i}\delta)$ and the source $f/\sqrt{\mathscr{E}_\delta}$ are invisible with respect to the wave measurement in the exterior of $B_{R'}$. In such a case, we say that cloaking due to anomalous localized resonance (CALR) occurs. For our subsequent study, we shall need relax a bit the requirements for resonance and ALR as follows. For a given sufficiently large $M\in\mathbb{R}_+$, if the configuration $(\epsilon_c,\epsilon_s+\mathrm{i}\delta, f)$ in \eqref{eq:helm1}--\eqref{eq:mc1} satisfies 
\begin{equation}\label{eq:res1}
\mathscr{E}_\delta[u]>M, 
\end{equation}
then we say that the configuration is {\it weakly resonant}; and if furthermore, \eqref{eq:res2} is also fulfilled, then we say that weak ALR occurs. The major reason for us to introduce the weak resonance and weak ALR is that our mathematical arguments are constructive, and we always construct the plasmonic structures which can make the corresponding energy dissipation large enough while the wave filed remain bounded outside a certain region, namely both \eqref{eq:res1} and \eqref{eq:res2} are fulfilled. From a practical point of view, the weak ALR should be enough for the cloaking application. In what follows, we shall not distinguish between ALR and weak ALR, and it should be clear from the context.

\subsection{Summary of the main results}

In the rest of this section, we briefly summarize the major findings of this paper for the convenience of readers. We consider two types of plasmonic structures, one with a core of the form \eqref{eq:mc1} and the other one with no core, i.e. $r_i=0$ in \eqref{eq:mc1}. Theorems~\ref{thm:main1} and \ref{thm:main2d1} contain the resonance results in 3D and 2D, respectively, for plasmonic structures with no cores. It is shown that in 3D, if the plasmon parameter is $\epsilon_s=-1-1/n_0$ with $n_0\in\mathbb{N}$ and $n_0\gg 1$ properly chosen, and in 2D, $\epsilon_s=-1$, and if the conditions \eqref{eq:k_no_three} for 3D and \eqref{eq:k_no_two} for 2D, are respectively fulfilled, then resonance occurs. Theorems~\ref{thm:CALR} and \ref{thm:CALR_two} give the ALR results in 3D and 2D, respectively, for core-shell plasmonic structures. In 3D, it is shown that if the medium parameters are chosen according to \eqref{eq:CALR3d1} and if the condition \eqref{eq:k_three} is fulfilled, then ALR occurs. Furthermore, there is a critical radius such that if the source is located outside the critical radius, then resonance does not occur. In 2D, it is shown that if $\epsilon_s=-1$ in the shell and the medium parameters fulfil the condition \eqref{eq:k_two}, then ALR occurs and moreover, there exists a critical radius such that if the source is located outside that radius, then resonance does not occur.

The outline of our paper is as follows. In Section 2, we consider the plasmon resonance and ALR results in three dimensions. Section 3 is devoted to the derivation of the spectral system of the Neumann-Poincar\'e operator as well as its application to the plasmon resonance study. Finally, in Section 4, we consider the plasmon resonance and ALR results in two dimensions.

\section{Plasmon resonance and ALR results in three dimensions}

In this section, we consider the resonance and ALR results associated with the Helmholtz system \eqref{eq:helm1}--\eqref{eq:mc1} in three dimensions. 

\subsection{Resonance result with no core in $\mathbb{R}^3$}

Assume that $B_{r_i}=\emptyset$. The Helmholtz system \eqref{eq:helm1}--\eqref{eq:mc1} can be simplified as the following transmission problem,
\begin{equation}\label{eq:general equation_without_core}
 \left\{
   \begin{array}{ll}
     \triangle u(x) + k_{1,\delta}^2 u(x) =0, & x\in B_{r_e},\medskip \\
     \triangle u(x) + k^2 u(x) =f(x), & x\in \mathbb{R}^3\backslash\overline{B_{r_e}}, \medskip\\
     {u|_-=u|_+, \quad (\epsilon_s+i\delta) \frac{\partial u}{\partial\nu}\big|_-= \frac{\partial u}{\partial\nu}\big|_+,} & x\in\partial B_{r_e},\medskip \\
   \displaystyle{\lim_{|x|\rightarrow+\infty} x\cdot\nabla u-\mathrm{i} k |x| u=0},
   \end{array}
 \right.
\end{equation}
where and throughout the rest of the paper,
\begin{equation}\label{eq:definition_k_1}
  k_{1,\delta}:=\frac{k}{\sqrt{\epsilon_s+i\delta}}\quad \mbox{with}\quad \Re{k_{1,\delta}}>0, \quad \mbox{and} \quad \Im{k_{1,\delta}}<0.
\end{equation}
It is supposed that the source $f(x)$ is supported outside a ball $B_{R_1}$ with $R_1>r_e$.

Let $j_n(t)$ and $h_n^{(1)}(t)$ be, respectively, the spherical Bessel and Hankel functions of order $n\in \mathbb{N}$, and $Y_n(\hat{x})$ be the spherical harmonics. Let $G(x)$ be the fundamental solution of the operator $\triangle+k^2$, namely
\begin{equation}\label{eq:fundamental_solution}
 G^{k}(x)=-\frac{e^{ik|x|}}{4\pi |x|}.
\end{equation}
The Newtonian potential of $f(x)$ is defined as
\begin{equation}\label{eq:potential_F}
  F(x)=\int_{\mathbb{R}^3} G^k(x-y)f(y)dy, \quad x\in\mathbb{R}^3,
\end{equation}
which verifies $(\triangle + k^2)F(x)=0,\ x\in B_{R_1}$. For $x\in B_{R_1}$, the potential $F(x)$ can be written as
\begin{equation}\label{eq:potential_F_expression}
 F(x)=\sum_{n=0}^{\infty} \beta_n j_n(kr)Y_n(\hat{x}).
\end{equation}
With the above preparations, the solution to \eqref{eq:general equation_without_core} can be written for $x\in B_{R_1}$ as
\begin{equation}\label{eq:solution_without_core_ge}
  u(x)=\left\{
         \begin{array}{ll}
          {\sum_{n=0}^{\infty} a_n j_n(k_{1,\delta}r) Y_n(\hat{x}),} & x\in B_{r_e}, \\
         { \sum_{n=0}^{\infty} b_n j_n(kr) Y_n(\hat{x}) + c_n h_n^{(1)}(kr) Y_n(\hat{x}) ,} & x\in B_{R_1}\backslash \overline{B_{r_e}}.
         \end{array}
       \right.
\end{equation}
Applying the third condition in \eqref{eq:general equation_without_core} on $\partial B_{r_e}$ to $u$ represented in \eqref{eq:solution_without_core_ge}, we have
\begin{equation}\label{eq:transmission_condition}
  \left\{
    \begin{array}{ll}
      a_n j_n(k_{1,\delta} r_e)=b_n j_n(kr_e) + c_n h_n^{(1)}(kr_e) ,\medskip \\
      \sqrt{\epsilon_s+i\delta}a_n j_n^{\prime}(k_{1,\delta} r_e)=b_n j_n^{\prime}(kr_e) + c_n  h_n^{(1)\prime}(kr_e).
    \end{array}
  \right.
\end{equation}
Solving the equations in \eqref{eq:transmission_condition}, we further have
\begin{equation}\label{eq:solution_without_core}
  \left\{
    \begin{array}{ll}
     \displaystyle{ a_n=b_n \frac{j_n^{\prime}(kr_e) h_n^{(1)}(kr_e)- h_n^{(1)\prime}(kr_e)j_n(kr_e)} {\sqrt{\epsilon_s+i\delta} j_n^{\prime}(k_{1,\delta} r_e)h_n^{(1)}(kr_e) - h_n^{(1)\prime}(kr_e)j_n(k_{1,\delta} r_e)  },}\medskip\\
    \displaystyle{  c_n=b_n \frac{j_n^{\prime}(kr_e) j_n(k_{1,\delta} r_e) -\sqrt{\epsilon_s+i\delta}j_n^{\prime}(k_{1,\delta} r_e)j_n(kr_e)}{\sqrt{\epsilon_s+i\delta} j_n^{\prime}(k_{1,\delta} r_e)h_n^{(1)}(kr_e) - h_n^{(1)\prime}(kr_e)j_n(k_{1,\delta} r_e)  }.}
    \end{array}
  \right.
\end{equation}
Since when $|x|>r_e$, $u(x)-F(x)$ satisfies
\begin{equation}\label{eq:r1}
(\triangle + k^2)(u(x)-F(x)) =0\quad\mbox{and}\quad  (u(x)-F(x))\rightarrow 0 \ \ \mbox{as}\ \ |x|\rightarrow \infty,
\end{equation}
one can show that there holds
\begin{equation}\label{eq:bn}
  b_n=\beta_n.
\end{equation}
Therefore, the solution to \eqref{eq:general equation_without_core} is given by \eqref{eq:solution_without_core_ge} with the coefficients defined in \eqref{eq:solution_without_core} and \eqref{eq:bn}.

 We are in a position to give the representation of the energy dissipation $\mathscr{E}_{\delta}[u]$. Direct calculations together with the help of Green's formula yield that
\begin{equation}\label{eq:en1}
 \begin{split}
   \mathscr{E}_{\delta}[u] & = \delta  \left( (k_{1,\delta})^2\int_{B_{r_e}}|u|^2 dx + \int_{\partial B_{r_e}}\frac{\partial u}{\partial\nu}\overline{u} ds(x)\right) \\
     & =\sum_{n=0}^{\infty}\delta |a_n|^2 \left( (k_{1,\delta})^2\int_{0}^{r_e}|j_n(k_{1,\delta}r)r|^2 dx + k_{1,\delta}r_e^2j_n^{\prime}(k_{1,\delta}r_e)\overline{j_n(k_{1,\delta}r_e)} \right).
 \end{split}
\end{equation}
Clearly, \eqref{eq:en1} indicates that if there exists $n_0\in\mathbb{N}$ such that
\begin{equation}\label{eq:condition_resonance}
 \delta |a_{n_0}|^2\rightarrow \infty\ \ \mbox{as}\ \ \delta\rightarrow\delta_0,
\end{equation}
then plasmon resonance occurs. To that end, we next analyze $a_n$. Using the fact
 \[
  j_n(t)h_n^{(1)\prime}(t)-j_n^{\prime}(t)h_n^{(1)}(t)=\frac{i}{t^2},
\]
one has from \eqref{eq:solution_without_core} and \eqref{eq:bn} that
\begin{equation}\label{eq:aaa1}
   a_n=\frac{1}{(kr_e)^2} \frac{-i \beta_n} {\sqrt{\epsilon_s+i\delta} j_n^{\prime}(k_{1,\delta} r_e)h_n^{(1)}(kr_e) - h_n^{(1)\prime}(kr_e)j_n(k_{1,\delta} r_e)  }.
\end{equation}
In order to make \eqref{eq:condition_resonance} happen, one should have by using \eqref{eq:aaa1} that
\begin{equation}\label{eq:fourier_resonance}
  \sqrt{\epsilon_s+i\delta} j_n^{\prime}(k_{1,\delta} r_e)h_n^{(1)}(kr_e) - h_n^{(1)\prime}(kr_e)j_n(k_{1,\delta} r_e)\rightarrow 0\ \ \mbox{as}\ \ \delta\rightarrow\delta_0.
\end{equation}
That is, we need to determine an appropriate plasmon material distribution $\epsilon_s+i\delta$ such that \eqref{eq:fourier_resonance} can occur. It is noted that \eqref{eq:fourier_resonance} is a nonlinear equation in terms of $\epsilon_s+i\delta$ and $n$. The rest of this subsection is devoted to analyzing this nonlinear equation.

\subsubsection{Simple numerical constructions}

Based on our earlier calculations, we first numerically solve the nonlinear equation \eqref{eq:fourier_resonance} to obtain some plasmonic structures that can induce the resonance. For simplicity, we set $k=1$ and $r_e=1$, and choose the source $f$ such that the Newtonian potential $F(x)$ in \eqref{eq:potential_F_expression} satisfy $\beta_n=1$ when $n=n_0$ and $\beta_n=0$ when $n\neq n_0$. 
Let the plasmon parameters be chosen as follows,
\begin{equation}
  \epsilon_s=\epsilon_{n_0} \quad \mbox{and} \quad \delta=\delta_{n_0},
\end{equation}
which depends on $n_0$. We numerically find the following parameters that can induce resonance,
\begin{equation}\label{eq:data}
  \left\{
    \begin{array}{ll}
      \epsilon_{n_0}=-1.303728$ \quad $\delta_{n_0}=0.498620 & \mbox{for}\quad n_0=1, \\
      \epsilon_{n_0}=-1.237160$ \quad $\delta_{n_0}=0.038434 & \mbox{for}\quad n_0=2, \\
      \epsilon_{n_0}=-1.224395$ \quad $\delta_{n_0}=0.001203 & \mbox{for}\quad n_0=3, \\
      \epsilon_{n_0}=-1.190550$ \quad $\delta_{n_0}=0.000019 & \mbox{for}\quad n_0=4.
    \end{array}
  \right.
\end{equation}
In Fig.~2, we plot the energy $\mathscr{E}_\delta$ against the change of the loss parameter $\delta$ when $\epsilon_{n_0}$ is fixed for $n_0=1,2,3,4$, which clearly demonstrate the resonance results at those critical values in \eqref{eq:data}. The numerical results also motivate us that as $n_0$ increases, the plasmon parameter $\epsilon_s$ approaches $-1$ and the loss parameter $\delta$ approaches $0$. 

\begin{figure}\label{fig:resoance}
  \centering
 {\includegraphics[width=4cm]{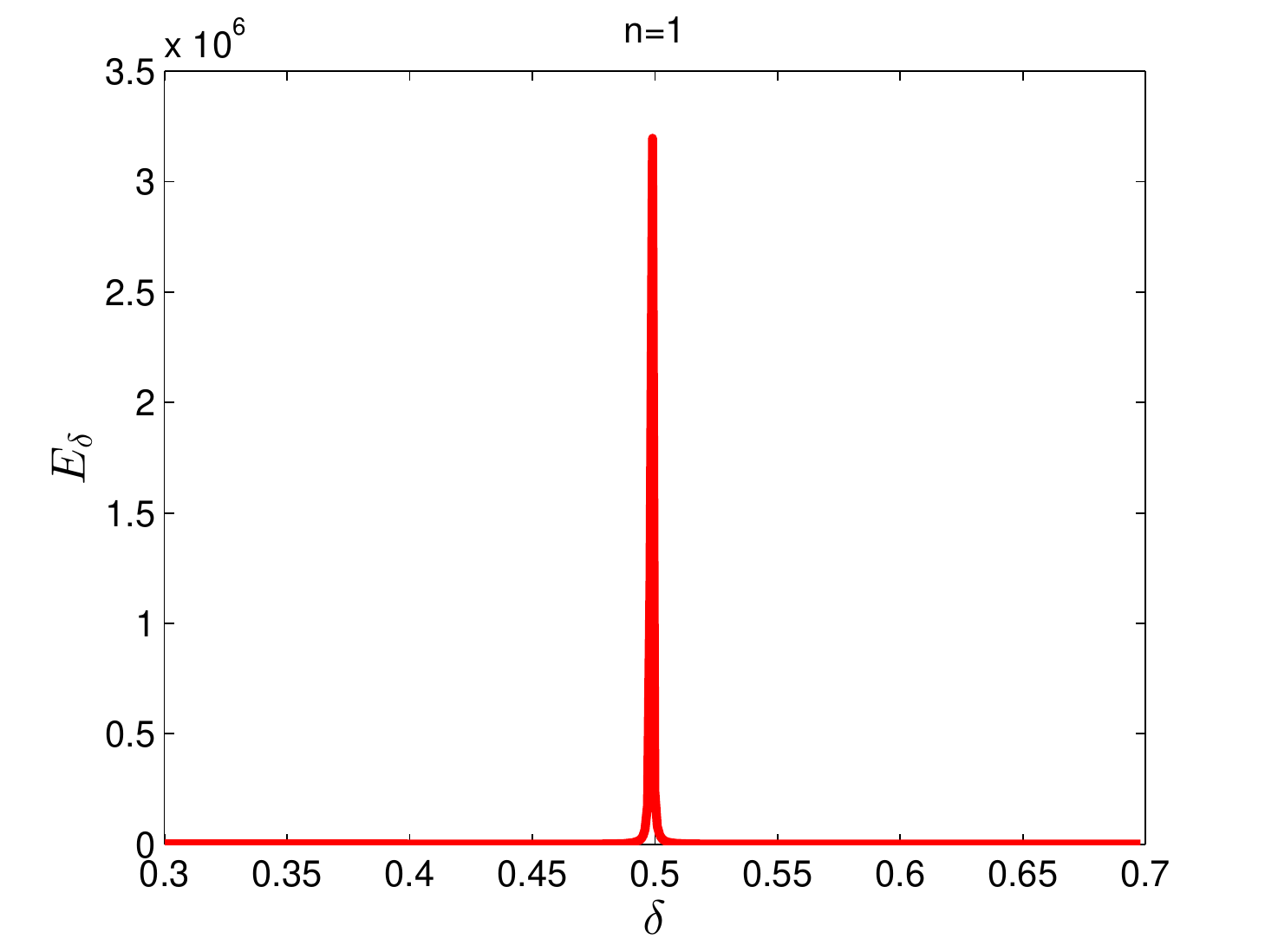}}
 {\includegraphics[width=4cm]{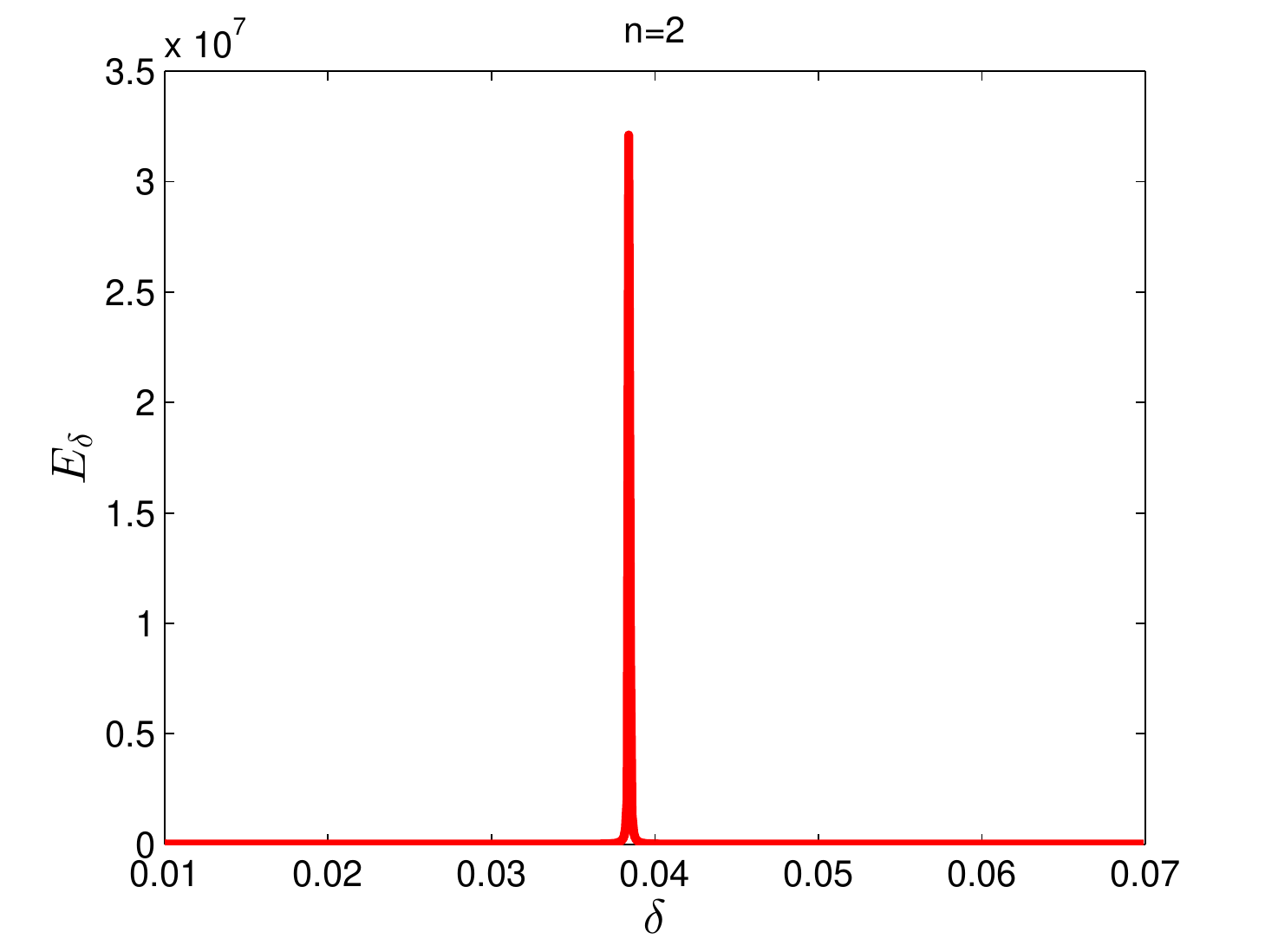}\\}
 {\includegraphics[width=4cm]{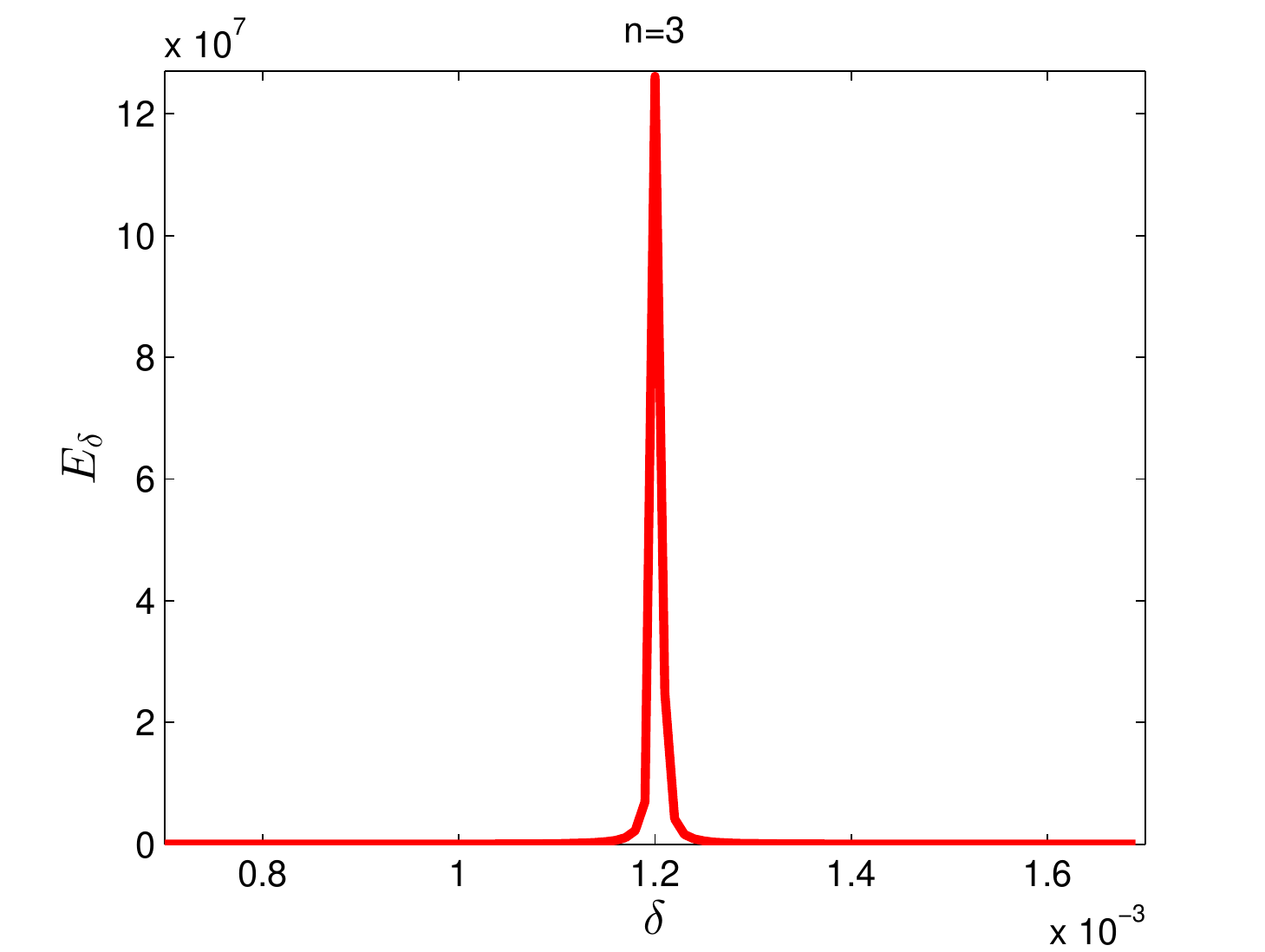}}
 {\includegraphics[width=4cm]{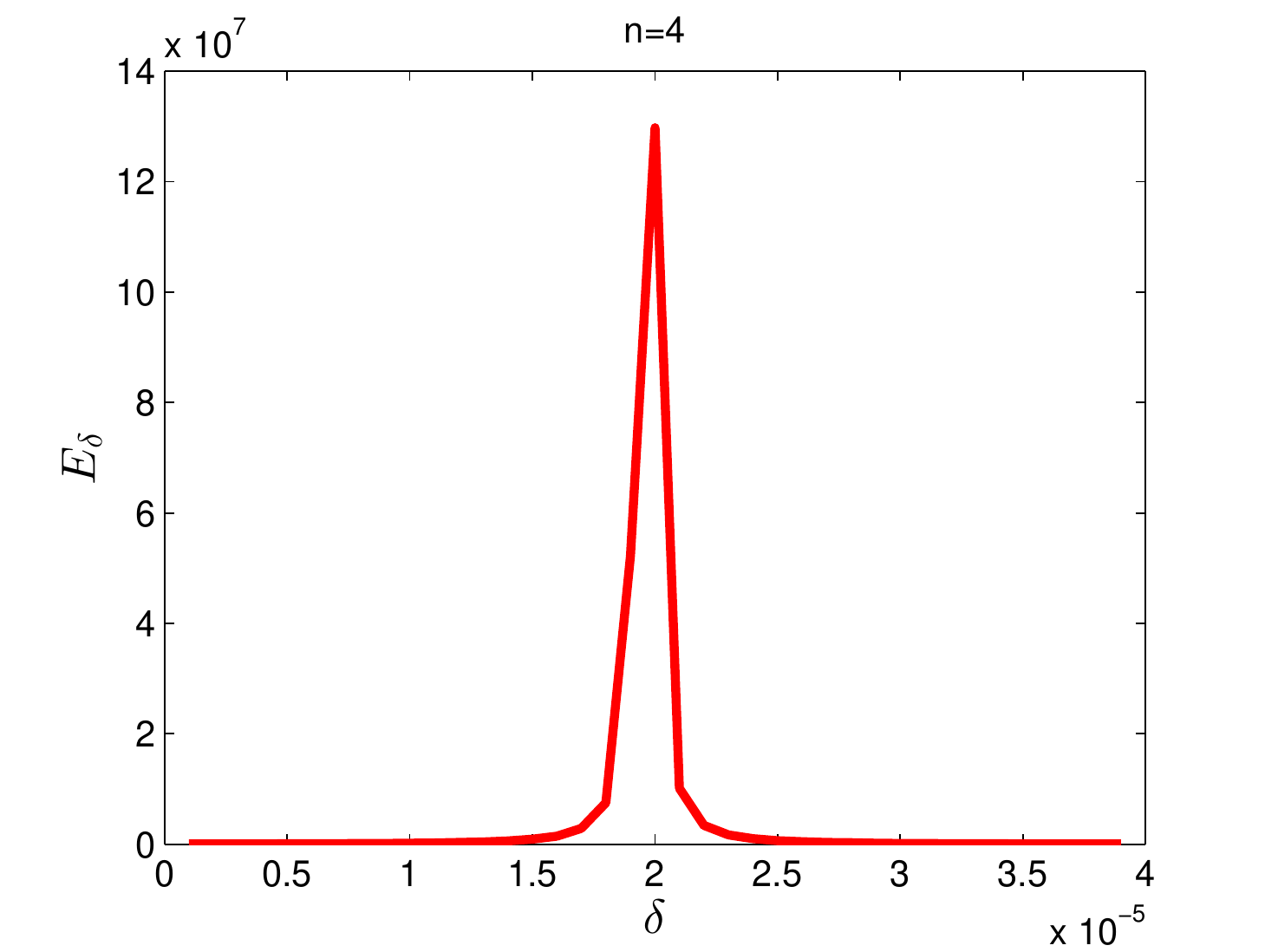}}
  \caption{The change of the dissipation energy $\mathscr{E}_\delta$ with respect to the change of $\delta$ when $\epsilon_{n_0}$ is fixed, $n_0=1,2,3,4$, in $\mathbb{R}^3$. }
\end{figure}

\subsubsection{General construction}

We next present the construction of a general plasmonic structure that can induce resonance. To that end, we first recall the following asymptotic properties of the spherical Bessel and Hankel functions, $j_n(t)$ and $h_n^{(1)}(t)$ for sufficiently large $n$ (cf. \cite{CK}),
\begin{equation}\label{eq:asymptotic_j}
  j_n(t)=\hat{j}_n(t)\left(1+ \check{j}_n(t) \right),\quad h_n^{(1)}(t)=\hat{h}_n^{(1)}(t) \left(1+\check{h}_n^{(1)}(t) \right),
\end{equation}
where  $\hat{f}_n(t)$, $\check{f}_n(t)$, $\hat{h}_n^{(1)}(t)$ and $\check{h}_n^{(1)}(t)$ are defined by
\begin{equation}\label{eq:dd1}
\hat{j}_n(t)=\frac{t^n}{(2n+1)!!}, \ \check{j}_n(t)=\mathcal{O}\left(\frac{1}{n}\right),\ \hat{h}_n^{(1)}(t)=\frac{(2n-1)!!}{it^{n+1}},  \ \check{h}_n^{(1)}(t)=\mathcal{O}\left(\frac{1}{n}\right).
\end{equation}
We have the following theorem regarding the plasmon resonance in three dimensions.
\begin{thm}\label{thm:main1}
Consider the Helmholtz system \eqref{eq:general equation_without_core}, where the Newtonian potential of the source $f$ is given in \eqref{eq:potential_F_expression}. Let $n_0\in\mathbb{N}$ fulfil the following two conditions:
 \begin{enumerate}
     \item $n_0$ is sufficiently large such that the asymptotic properties of $j_{n}(t)$ and $h_{n}^{(1)}(t)$ in \eqref{eq:asymptotic_j} hold for $n\geq n_0$;
     \item $\beta_{n_0}\neq 0$.
   \end{enumerate}
Let the plasmon parameters be chosen of the following form
  \begin{equation}\label{eq:ccc1}
    \epsilon_s=-1-\frac{1}{n_0} ,\quad \delta\in\mathbb{R}_+ \quad \mbox{and} \quad \delta\ll1. 
  \end{equation}
  Then if the plasmon configuration fulfils the following condition, 
  \begin{equation}\label{eq:k_no_three}
 \hat{j}_{n_0}(k_{1,\delta} r_e) \hat{h}_{n_0}^{(1)}(kr_e)  \left(  \sqrt{\epsilon_s+i\delta}  \check{j}_{n_0}(t)^{\prime}(k_{1,\delta} r_e)\left(1+\check{h}_{n_0}^{(1)}(kr_e)\right)  -\check{h}_{n_0}^{(1)\prime}(kr_e)\left(1+\check{j}_{n_0}(k_{1,\delta} r_e)\right)   \right)=0,  
\end{equation}
one has that $\mathscr{E}_\delta[u]\sim \delta^{-1}$. 
\end{thm}
\begin{proof}
Suppose that $\epsilon_s=-1-1/n_0$ with $n_0$ satisfying the two conditions stated in the theorem. By straightforward calculations, we first have that
\begin{equation}
  \left|\sqrt{\epsilon_s+i\delta} j_{n_0}^{\prime}(k_{1,\delta}r_e)h_{n_0}^{(1)}(kr_e) - h_{n_0}^{(1)\prime}(kr_e)j_{n_0}(k_{1,\delta}r_e) \right|\approx \delta \left(1+\mathcal{O}\left(\frac{1}{n_0}\right) \right).
\end{equation}
From the solution given in \eqref{eq:solution_without_core} and with the help of Green's formula, one has that
\[
 \begin{split}
   & \mathscr{E}_{\delta}[u] =\delta\int_{B_{r_e}}|\nabla u|^2dx = \delta k_{1,\delta}^2\int_{B_{r_e}}|u|^2dx + \delta\int_{\partial B_{r_e}} \frac{\partial u}{\partial \nu} \overline{u}ds(x)\\
     & \geq \delta k_{1,\delta}^2\int_{B_{r_e}}|a_{n_0} j_{n_0}(k_{1,\delta}r) Y_{n_0}(\hat{x})|^2dx \\
     &  +\delta\int_{\partial B_{r_e}} a_{n_0} k_{1,\delta} j_{n_0}^{\prime}(k_{1,\delta}r_e)  \overline{ \left(a_{n_0}   j_{n_0}(k_{1,\delta}r_e) \right)} |Y_{n_0}(\hat{x})|^2 ds(x) \approx \frac{|\beta_{n_0}|^2 n_0}{\delta}\left(1+\mathcal{O}\left(\frac{1}{n_0}\right) \right),
 \end{split}
\]
which readily completes the by noting that $\beta_{n_0}\neq 0$.
\end{proof}
\begin{rem}\label{rem:2.1}
Since $\delta\ll 1$, we see that the plasmon configuration in Theorem~\ref{thm:main1} induces resonance. The next thing one needs to verify is that the equation \eqref{eq:k_no_three} yields a nonempty set of parameters. This is indeed the case and we next present a numerical example for illustration. We set $r_e=1, n_0=500, \epsilon_s=1-1/n_0$ and $\delta=0.5^{n_0}$, and let $k$ be a free parameter. Fig.~3 plots the quantity in the LHS of \eqref{eq:k_no_three} against $k$ around $k=8$. One readily sees that there do exist $k$'s such that \eqref{eq:k_no_three} holds. Hence, resonance occurs with the aforesaid parameters at those $k$'s. One can also fix $k$ and determine the other parameters by solving \eqref{eq:k_no_three}. 
\end{rem}

\begin{figure}[t]
  \centering
 {\includegraphics[width=5cm]{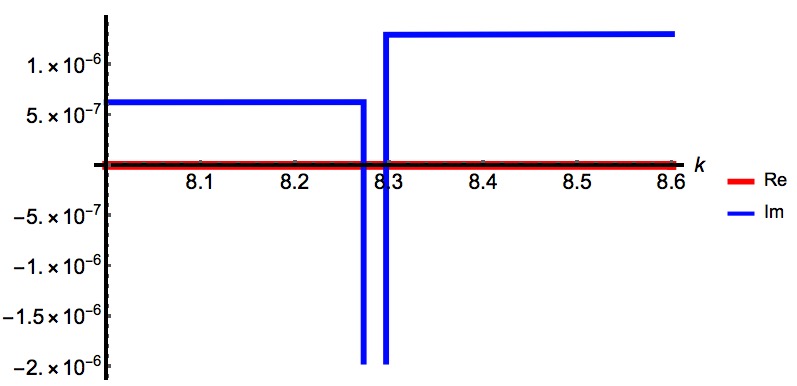}}
  \caption{The real and imaginary parts of the LHS quantity in \eqref{eq:k_no_three} with respect the change of the wavenumber $k$. }
  \label{fig:3}
\end{figure}

\subsection{ALR result with a core in $\mathbb{R}^3$}

In this subsection, we consider the Helmholtz system \eqref{eq:helm1}--\eqref{eq:mc1} in $\mathbb{R}^3$ with $r_i\neq 0$, and show that ALR can be induced. Let the Newtonian potential of the source $f$ is again given in \eqref{eq:potential_F_expression}. The solution to \eqref{eq:helm1}--\eqref{eq:mc1} in $B_{R_1}$ can be expressed as
\begin{equation}\label{eq:e1}
  u(x)=\left\{
         \begin{array}{ll}
          {\sum_{n=0}^{\infty} a_n j_n(\frac{kr}{\sqrt{\epsilon_c}}) Y_n(\hat{x})}, & x\in B_{r_i}, \\
          {\sum_{n=0}^{\infty} b_n j_n(k_{1,\delta}r) Y_n(\hat{x}) + c_n h_n^{(1)}(k_{1,\delta}r) Y_n(\hat{x})} , & x\in B_{r_e}\backslash B_{r_i}  \\
          {\sum_{n=0}^{\infty} e_n j_n(kr) Y_n(\hat{x}) + d_n h_n^{(1)}(kr) Y_n(\hat{x})} , & x\in B_{R_1}\backslash B_{r_e}. 
         \end{array}
       \right.
\end{equation}
By applying the transmission conditions across $\partial B_{r_i}$ and $\partial B_{r_e}$, we can have from \eqref{eq:e1} that
\begin{equation}\label{eq:transmission_condition_with_core_3d}
  \left\{
    \begin{array}{ll}
     { a_n j_n(\frac{kr_i}{\sqrt{\epsilon_c}})=b_n j_n(k_{1,\delta}r_i) + c_n h_n^{(1)}(k_{1,\delta}r_i) }, \\
     { \sqrt{\epsilon_c}a_n j_n^{\prime}(\frac{kr_i}{\sqrt{\epsilon_c}})=\sqrt{\epsilon_s+i\delta} \left(b_n j_n^{\prime}(k_{1,\delta}r_i) + c_n  h_n^{(1)\prime}(k_{1,\delta}r_i) \right),} \\
    {  b_n j_n(k_{1,\delta}r_e) + c_n h_n^{(1)}(k_{1,\delta}r_e) =  e_n j_n(kr_e)  + d_n h_n^{(1)}(kr_e),} \\
     { \sqrt{\epsilon_s+i\delta} \left(b_n j_n^{\prime}(k_{1,\delta}r_e) + c_n  h_n^{(1)\prime}(k_{1,\delta}r_e) \right)= e_n j_n^{\prime}(kr_e) + d_n  h_n^{(1)\prime}(kr_e).}
    \end{array}
  \right.
\end{equation}
Solving the equations in \eqref{eq:transmission_condition_with_core_3d}, one has that $ a_n={\tilde{a}_n}/{g_n }, b_n={\tilde{b}_n}/{g_n }, c_n={\tilde{c}_n}/{g_n }, d_n={\tilde{d}_n}/{g_n }$, where
\begin{align}
\tilde{a}_n= & e_n \sqrt{\epsilon_s+i\delta}\times \zeta_n \times \left(j_n^{\prime}(k_{1,\delta}r_i) h_n^{(1)}(k_{1,\delta}r_i) - h_n^{(1)\prime}(k_{1,\delta}r_i)j_n(k_{1,\delta}r_i)\right),\\
    \tilde{b}_n= & e_n \times\zeta_n \times \left(\sqrt{\epsilon_c}j_n^{\prime}(\frac{kr_i}{\sqrt{\epsilon_c}}) h_n^{(1)}(k_{1,\delta}r_i) - \sqrt{\epsilon_s+i\delta} h_n^{(1)\prime}(k_{1,\delta}r_i)j_n(\frac{kr_i}{\sqrt{\epsilon_c}})\right),\\
    \tilde{c}_n= & e_n \times \zeta_n \times \left(\sqrt{\epsilon_s+i\delta} j_n^{\prime}(k_{1,\delta}r_i)j_n(\frac{kr_i}{\sqrt{\epsilon_c}}) - \sqrt{\epsilon_c}j_n^{\prime}(\frac{kr_i}{\sqrt{\epsilon_c}}) j_n(k_{1,\delta}r_i) \right),\\
   \zeta_n:= &  j_n^{\prime}(kr_e) h_n^{(1)}(kr_e) - h_n^{(1)\prime}(kr_e)j_n(kr_e),
   \end{align}
 and
 \begin{align}
 \tilde{d}_n= &e_n \left( \sqrt{\epsilon_c} \left(h_n^{(1)}(k_{1,\delta}r_i)j_n(k_{1,\delta}r_e)- h_n^{(1)}(k_{1,\delta}r_e)j_n(k_{1,\delta}r_i)\right) j_n^{\prime}(\frac{kr_i}{\sqrt{\epsilon_c}})j_n^{\prime}(kr_e) + \right.\nonumber \\
    & (\epsilon_s+i\delta)\left( h_n^{(1)\prime}(k_{1,\delta}r_i)j_n^{\prime}(k_{1,\delta}r_e) -h_n^{(1)\prime}(k_{1,\delta}r_e)j_n^{\prime}(k_{1,\delta}r_i) \right) j_n(\frac{kr_i}{\sqrt{\epsilon_c}})j_n(kr_e) + \nonumber\\
    & \sqrt{\epsilon_s+i\delta}\left( j_n^{\prime}(k_{1,\delta}r_i)h_n^{(1)}(k_{1,\delta}r_e) - h_n^{(1)\prime}(k_{1,\delta}r_i)j_n(k_{1,\delta}r_e) \right) j_n(\frac{kr_i}{\sqrt{\epsilon_c}})j_n^{\prime}(kr_e) +\nonumber \\
    & \left. \sqrt{\epsilon_s+i\delta}\sqrt{\epsilon_c} \left( h_n^{(1)\prime}(k_{1,\delta}r_e)j_n(k_{1,\delta}r_i) - j_n^{\prime}(k_{1,\delta}r_e)h_n^{(1)}(k_{1,\delta}r_i)\right) j_n^{\prime}(\frac{kr_i}{\sqrt{\epsilon_c}})j_n(kr_e) \right),
 \end{align}
 \begin{align}
  g_n = & (\epsilon_s+i\delta)\left( h_n^{(1)\prime}(k_{1,\delta}r_e)j_n^{\prime}(k_{1,\delta}r_i) -h_n^{(1)\prime}(k_{1,\delta}r_i)j_n^{\prime}(k_{1,\delta}r_e) \right)h_n^{(1)}(kr_e) j_n(\frac{kr_i}{\sqrt{\epsilon_c}}) +\nonumber  \\
      & \sqrt{\epsilon_s+i\delta}\left(h_n^{(1)\prime}(k_{1,\delta}r_i)j_n(k_{1,\delta}r_e)  -j_n^{\prime}(k_{1,\delta}r_i) h_n^{(1)}(k_{1,\delta}r_e) \right)h_n^{(1)\prime}(kr_e) j_n(\frac{kr_i}{\sqrt{\epsilon_c}}) +\nonumber \\
      & \sqrt{\epsilon_c}\left( h_n^{(1)}(k_{1,\delta}r_e)j_n(k_{1,\delta}r_i) -h_n^{(1)}(k_{1,\delta}r_i)j_n(k_{1,\delta}r_e) \right)j_n^{\prime}(\frac{kr_i}{\sqrt{\epsilon_c}})h_n^{(1)\prime}(kr_e) +\nonumber \\
      &  \sqrt{\epsilon_c}\sqrt{\epsilon_s+i\delta}\left( j_n^{\prime}(k_{1,\delta}r_e) h_n^{(1)}(k_{1,\delta}r_i) - h_n^{(1)\prime}(k_{1,\delta}r_e)j_n(k_{1,\delta}r_i)    \right) j_n^{\prime}(\frac{kr_i}{\sqrt{\epsilon_c}})h_n^{(1)}(kr_e),\label{eq:2.29}
 \end{align}
with $k_{1,\delta}$ given in \eqref{eq:definition_k_1}.

By a similar reasoning to that in \eqref{eq:r1} and \eqref{eq:bn}, one can show that $e_n=\beta_n$. With the above series representation of the solution, we next consider the occurrence of ALR, namely configurations to make both \eqref{eq:res1} and \eqref{eq:res2} fulfilled. To that end, we need to impose a certain constraint on the source $f(x)$ such that the potential $F(x)$ is of the following form
\begin{equation}\label{eq:potential_F_constain}
 F(x)=\sum_{n=N}^{\infty} \beta_n j_n(kr)Y_n(\hat{x}),
\end{equation}
for some sufficiently large $N$ so that when $n\geq N$,  the asymptotic properties in \eqref{eq:asymptotic_j} hold. In the following, in order to simplify the statement, we also need to introduce the following two functions:
\begin{equation*}
\begin{split}
& \varphi_1(n,b_1,b_2,r_1,r_2)=\hat{j}_n(r_1) \hat{h}_n^{(1)}(r_2)  \left( b_1 \check{j}_n^{\prime}(r_1)\left(1+\check{h}_n^{(1)}(r_2)\right)   -b_2\check{h}_n^{(1)\prime}(r_2)\left(1+\check{j}_n(r_1)\right) \right),\\
& \varphi_2(n,b_1,b_2,r_1,r_2)= b_1 j_n^{\prime}(r_1) h_n^{(1)}(r_2) - b_2 h_n^{(1)\prime}(r_2) j_n(r_1) , 
\end{split}
\end{equation*}
where  $\hat{f}_n(t)$, $\check{f}_n(t)$, $\hat{h}_n^{(1)}(t)$ and $\check{h}_n^{(1)}(t)$ are defined in \eqref{eq:asymptotic_j}--\eqref{eq:dd1}. 

 \begin{thm}\label{thm:CALR}
 Consider the Helmholtz system \eqref{eq:helm1}--\eqref{eq:mc1} in $\mathbb{R}^3$ with the Newtonian potential $F$ of the source $f$ satisfying \eqref{eq:potential_F_constain}. Let the plasmon configuration be chosen of the following form,
  \begin{equation}\label{eq:CALR3d1}
  \epsilon_c=(1+1/n_0)^2 , \quad  \epsilon_s=-1-1/n_0 \quad \mbox{and}  \quad \delta= \rho^{n_0},
 \end{equation}
 where $\rho:=r_i/r_e<1$ and $n_0\in\mathbb{N}$ with $n_0\gg 1$. If the plasmon configuration fulfils the following condition,
\begin{equation}\label{eq:k_three}
\begin{split}
&\varphi_1(n_0,\sqrt{\epsilon_c},\tau_{s,\delta},\frac{kr_i}{\sqrt{\epsilon_c}},k_{1,\delta} r_i)\varphi_2(n_0,\tau_{s,\delta},1,k_{1,\delta} r_e,kr_e) + \varphi_1(n_0,\tau_{s,\delta},1,k_{1,\delta} r_e, kr_e)\\
&\times\left(\varphi_2(n_0,\sqrt{\epsilon_c},\tau_{s,\delta},\frac{kr_i}{\sqrt{\epsilon_c}},k_{1,\delta} r_i) - \varphi_1(n_0,\sqrt{\epsilon_c},\tau_{s,\delta},\frac{kr_i}{\sqrt{\epsilon_c}},k_{1,\delta} r_i)   \right)=0,
\end{split}
\end{equation}
where $\tau_{s,\delta}:=\sqrt{\epsilon_s+i\delta}$, then there is a critical radius $r_*:=\sqrt{r_e^3/r_i}$ such that if $f$ lies within this radius, $\mathscr{E}_{\delta}[u]\geq \mu_0^{n_0} n_0$ with $\mu_0>1$, and $u(x)$ remains bounded for $|x|>r_e^2/r_i$; and if $f$ lies outside this radius, $\mathscr{E}_{\delta}[u]$ is bounded by a constant depending only on $f, k$ and $r_e$.  

\end{thm}

\begin{proof}
For notational convenience of the proof, we set $\tilde\beta_n:=\frac{\beta_n}{(2n+1)!!}$, $n\geq N$. By \eqref{eq:CALR3d1} and \eqref{eq:k_three}, together with the use of \eqref{eq:asymptotic_j} one can derive the following estimates when $n=n_0$,
\begin{equation}\label{eq:coefficient_g_n_0}
  g_{n_0}\approx \delta^2 + \rho^{2n_0},\quad  \tilde{b}_{n_0}\approx i\delta \beta_{n_0},
\end{equation}
\begin{equation}\label{eq:coefficient_c_n_0}
 \tilde{c}_{n_0}\approx \frac{n_0(kr_i)^{2n_0}}{(1\cdot 3 \cdots (2n_0+1))^2} \beta_{n_0},\quad  \tilde{d}_{n_0}\approx \frac{\delta n_0(kr_e)^{2n_0}}{(1\cdot 3 \cdots (2n_0+1))^2} \beta_{n_0},
\end{equation}
and when $n\neq n_0$,
\begin{equation}\label{eq:coefficient_g_n}
  g_n\approx  \frac{(n-n_0)^2}{n^2 n_0^2},\quad  b_n\approx \frac{n-n_0}{nn_0}\beta_n,
\end{equation}
\begin{equation}\label{eq:coefficient_d_n}
  c_n\approx \frac{n(kr_i)^{2n} \beta_n}{(1\cdot 3 \cdots (2n+1))^2},\quad  \tilde{d}_n \approx  \frac{(kr_e)^{2n}\beta_n }{(1\cdot 3 \cdots (2n+1))^2}  \frac{n-n_0}{n_0},
\end{equation}
Noting $\delta=\rho^{n_0}$, from \eqref{eq:coefficient_g_n_0} and \eqref{eq:coefficient_c_n_0} and by direct calculations, one can show that
\begin{equation}\label{eq:energy_estimate_with_core}
   \mathscr{E}_{\delta}[u]\geq\frac{n_0\tilde{\beta}_{n_0}^2 (kr_e)^{2n_0} \delta } {\delta^2+\rho^{2n_0}}\geq n_0\tilde{\beta}_{n_0}^2 \left( \frac{k^2r_e^3}{r_i} \right)^{n_0}.
\end{equation}
Consider first that the source $f(x)$ is supported inside the critical radius $r_*=\sqrt{r_e^3/r_i}$. By \eqref{eq:potential_F_constain} and the asymptotic property of $j_n(t)$ in \eqref{eq:asymptotic_j}, one can verify that there exists $\tau\in\mathbb{R}_+$ such that
\begin{equation}\label{eq:lll2}
 \limsup_{n\rightarrow\infty} (\tilde{\beta}_n)^{1/n}=\sqrt{\frac{r_i}{k^2r_e^3}+\tau}.
\end{equation}
Combining \eqref{eq:lll2} and \eqref{eq:energy_estimate_with_core}, one can obtain that
\[
 \mathscr{E}_{\delta}[u]\geq n_0 \left(\frac{r_i}{k^2r_e^3}+\tau\right)^{n_0} \left( \frac{k^2r_e^3}{r_i} \right)^{n_0}.
\]
Next, we suppose that the source is supported outside the critical radius $r_*$. Then there exists $\eta>0$ such that
\[
 \limsup_{n\rightarrow\infty} (\tilde{\beta}_n)^{1/n}\leq\frac{1}{kr_*+\eta},
\]
and the dissipation energy $\mathscr{E}_{\delta}$ can be estimated as follows
\[
 \begin{split}
   \mathscr{E}_{\delta}[u] & \approx \sum_{n\geq N,n\neq n_0} n^3 \left(\frac{n_0}{n-n_0}\right)^2 (kr_e)^{2n}\tilde{\beta}_n^2 + \frac{n_0\tilde{\beta}_{n_0}^2 (kr_e)^{2n_0} \delta } {\delta^2+\rho^{2n_0}}, \\
     & \leq \sum_{n\geq N,n\neq n_0} n^3 \left(\frac{n_0}{n-n_0}\right)^2 \rho^n +  n_0\tilde{\beta}_{n_0}^2 \left( \frac{k^2r_e^3}{r_i} \right)^{n_0}\leq C, 
 \end{split}
\]
which means that resonance does not occur.

Next we prove the boundedness of the solution $u(x)$ when $|x|>r_e^2/r_i$. From \eqref{eq:coefficient_g_n_0} and \eqref{eq:coefficient_c_n_0}, when $n=n_0$ one has that
\begin{equation}\label{eq:estimate_d_n0}
 \begin{split}
   |d_{n_0}h_{n_0}^{(1)}(kr)| & \leq \frac{|\tilde{\beta}_{n_0}| (kr_e)^{2n_0}n_0}{1\cdot 3 \cdots (2n_0+1)}\frac{\delta}{\delta^2+\rho^{2n_0}}\frac{1\cdot 3 \cdots (2n_0-1)}{(kr)^{n_0+1}} \\
     & \leq |\tilde{\beta}_{n_0}|(kr_e)^{n_0} \left(\frac{r_e^2}{r_i}\right)^{n_0}\frac{1}{r^{n_0}} ,
 \end{split}
\end{equation}
and when $n\neq n_0$, one can obtain from \eqref{eq:coefficient_g_n} and \eqref{eq:coefficient_d_n},
\begin{equation}\label{eq:estimate_d_n}
 \begin{split}
   |d_n h_{n}^{(1)}(kr)| & \leq  \frac{|\tilde{\beta}_n| (kr_e)^{2n}}{1\cdot 3 \cdots (2n+1)}  \frac{n^2n_0}{n-n_0} \frac{1\cdot 3 \cdots (2n-1)}{(kr)^{n+1}},\\
     & \leq |\tilde{\beta}_{n}|(kr_e)^{n} \frac{nn_0}{n-n_0}\frac{r_e^n }{r^n}.
 \end{split}
\end{equation}
Hence from \eqref{eq:estimate_d_n0} and \eqref{eq:estimate_d_n}, one has that
\begin{equation}
  |u(x)-F(x)|\leq \sum_{n\geq N} |\tilde{\beta}_{n}|(kr_e)^{n}\leq C, \quad \mbox{when} \quad |x|\geq r_e^2/r_i.
\end{equation}

The proof is complete.
\end{proof}

\begin{rem}
Similar to Remark~\ref{rem:2.1}, we can numerically verify the equation in \eqref{eq:k_three} yields a nonempty set of parameters. For illustration, we set  $n_0=300, r_i=0.5, r_e=1, \epsilon_c=(1+1/n_0)^2, \epsilon_s=-1-1/n_0$ and $\delta=(r_i/r_e)^{n_0}$, and let $k$ be a free parameter. Fig.~4 plots the quantity in the LHS of \eqref{eq:k_three} against $k$ over an interval. It can be seen that there do exist $k$'s such that \eqref{eq:k_three} holds. 
\end{rem}
\begin{figure}[t]
  \centering
 {\includegraphics[width=5cm]{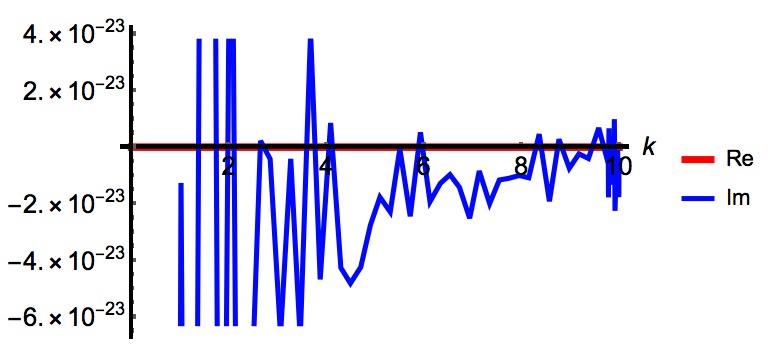}}
  \caption{The real and imaginary parts of the LHS quantity in \eqref{eq:k_three} with respect the change of the wavenumber $k$. }
  \label{fig:4}
\end{figure}

\section{Spectral system of the Neumann-Poinc\'are operator and its application to plasmon resonance in $\mathbb{R}^3$}

With the help of the results obtained in the previous section, we are able to derive the complete spectral syestem of the Neumann-Poinc\'are operator associated with the Helmholtz equation at finite frequencies in the radial geometry. The spectral result is of independent mathematical interest and it can also help us to derive the resonance and ALR results in in Section 2.

\subsection{Spectral system of the Neumann-Poinc\'are operator}

 First we introduce the single layer potential that is defined as
\[
 S_{\partial B_R}^{k}[\varphi](x)=\int_{\partial B_R}G^k(x-y)\varphi(y)ds(y), \quad x\in\mathbb{R}^3,
\]
where $\varphi(x)\in L^2(\partial B_R)$ and $G^k(x)$ is given in \eqref{eq:fundamental_solution}. The conormal derivative of the single layer potential enjoys the following jump formula
\begin{equation}\label{eq:jump_relation}
\frac{\partial}{\partial\nu}S_{\partial B_R}^{k}[\varphi]|_{\pm}(x)=\left(\pm\frac{1}{2}I + \left(K_{\partial B_R}^k\right)^* \right)[\varphi](x), \quad x\in\partial B_{R},
\end{equation}
where
\[
\left(K_{\partial B_R}^k\right)^*[\varphi](x)= \mbox{p.v.}\int_{\partial B_R} \frac{\partial}{\partial\nu_x}G^k(x-y)\varphi(y)ds(y), \quad x\in\partial B_{R},
\]
which is referred to as the Neumann-Poinc\'are operator. In \eqref{eq:jump_relation}, the $\pm$ indicate the limit (to $\partial B_R$) from outside and inside of $B_R$, respectively. It is remarked that $(K^k_{\partial B_R})^*$ is compact but non self-adjoint. In \cite{Ack13}, it is shown for the quasistatic limit with $k=0$, $(K^0_{\partial B_R})^*$ is symmetrizable, but it is not the case with $k\neq 0$. Before giving the eigenvalues of the operator $\left(K_{\partial B_R}^k\right)^*$, we first introduce the following identity.
\begin{lem}
\begin{equation}\label{eq:k_star_identity}
  \left(K_{\partial B_R}^k\right)^*[\varphi](x)=-\frac{1}{2R}S_{\partial B_R}^{k}[\varphi](x)- \frac{ik}{8\pi R}\int_{\partial B_R}e^{ik|x-y|}\varphi(y)ds(y),\quad x\in\partial B_R.
\end{equation}
\end{lem}
\begin{proof}
The identify can be easily verified by noting that for $x$ and $y$ both belonging to $\partial B_R$, one has
\[
\begin{split}
   \frac{\partial}{\partial\nu_x}G^k(x-y)= & \frac{-e^{ik|x-y|}}{4\pi}\left(-\frac{\langle x-y,\nu_x\rangle}{|x-y|^3} + ik\frac{\langle x-y,\nu_x\rangle}{|x-y|^2}  \right)\\
    & = \frac{-e^{ik|x-y|}}{4\pi}\left(-\frac{1}{2R|x-y|} + \frac{ik}{2R}  \right),
\end{split}
\]
where the following fact is used
\[
  \frac{\langle x-y,\nu_x\rangle}{|x-y|^2}=\frac{R(1-\langle \nu_y,\nu_x\rangle)}{2R^2(1-\langle \nu_y,\nu_x\rangle)}=\frac{1}{2R}.
\]
\end{proof}

The spectral system of the operator $\left(K_{\partial B_R}^k\right)^*$ is contained in the following theorem.

\begin{thm}\label{thm:eigenvalue_k-star}
Let $k$ and $R$ satisfy the following condition
\begin{equation}\label{eq:assumption_kR}
  j_n(kR)\neq j_{n+2}(kR) \quad \mbox{for} \quad n\geq 0.
\end{equation}
Then
\begin{equation}
  \left(K_{\partial B_R}^k\right)^*[Y_m](x)=\lambda_{m,k,R} Y_m(\hat{x}),
\end{equation}
where
\begin{equation}\label{eq:eigenvalue_lambda_m}
  \lambda_{m,k,R}=\frac{1}{2}-i k^2 R^2 j_n^{\prime}(kR) h_n^{(1)}(kR)=-\frac{1}{2}-i k^2 R^2 j_n(kR) h_n^{(1)\prime}(kR) ,
\end{equation}
and on the boundary $\partial B_R$,
\[
   S_{\partial B_R}^{k}[Y_m](x)=\chi_{m,k,R} Y_m(\hat{x}), \quad x\in \partial B_R,
\]
 where $\chi_{m,k,R}=-i k R^2 h_m^{(1)}(kR) j_m(kR)$.
\end{thm}

\begin{proof}
By the Funk-Hecke formula (cf. \cite{CK}), one has that
\begin{equation}\label{eq:F-H_formula}
 \int_{\partial B_R}e^{ik|x-y|}Y_m(\hat{y})ds(y)=2\pi R^2 E_{m,k,R} Y_m(\hat{x}),
\end{equation}
where
\begin{equation}\label{eq:expression_E_m}
 E_{m,k,R}=\int_{-1}^{1}e^{i\sqrt{2}kR\sqrt{1-t}}P_m(t)dt,
\end{equation}
with $P_m(t)$ denoting the Legendre polynomial of degree $m$. Since $S_{\partial B_R}^{k}[Y_m](x)$ satisfies
\[
 (\triangle + k^2) S_{\partial B_R}^{k}[Y_m](x)=0 \quad x\in B_R,
\]
we can assume that
\begin{equation}\label{eq:Y_m_expasion}
  S_{\partial B_R}^{k}[Y_m](x)=\sum_{n=0}^{\infty} \gamma_{n,k,R} j_n(kr) Y_n(\hat{x}) \quad x\in B_R.
\end{equation}
Substituting \eqref{eq:k_star_identity} and \eqref{eq:Y_m_expasion} into the jump formula \eqref{eq:jump_relation}, together with the help of \eqref{eq:F-H_formula}, one can obtain that
\begin{equation}\label{eq:identigy_1}
  \sum_{n=0}^{\infty} \gamma_{n,k,R}  \left(k j_n^{\prime}(kR) + \frac{j_n(kR)}{2R} \right) Y_n(\hat{x})=\left( -\frac{1}{2}-\frac{ikR}{4}E_m \right)Y_m(\hat{x}).
\end{equation}
From the two properties of the spherical Bessel function $j_n(t)$ that
\begin{equation}
\begin{split}
    & j_{n+1}(t)=-t^n\frac{d}{dt}\left( t^{-n} f_n(t) \right),\quad n=0,1,\cdots,  \\
    & j_{n+1}(t)+j_{n-1}(t)=\frac{2n+1}{t}f_n(t), \quad n=1,2,\cdots,
\end{split}
\end{equation}
and together with the assumption in \eqref{eq:assumption_kR}, one can show that
\begin{equation}
  \left(k j_n^{\prime}(kR) + \frac{j_n(kR)}{2R} \right)\neq 0, \quad \mbox{for}\quad n\geq 0.
\end{equation}
Therefore one can obtain from the equation \eqref{eq:identigy_1} that
\begin{equation}\label{eq:coeff_gamma}
  \gamma_{n,k,R}=
\left\{
  \begin{array}{ll}
    \displaystyle{-\frac{(2+ikRE_m)R}{2(2kRj_n^{\prime}(kR)+ j_n(kR))},} & n=m,\medskip \\
    \ \ \ 0, & n\neq m.
  \end{array}
\right.
\end{equation}
Finally, from the equation \eqref{eq:k_star_identity}, one has that
\[
  \left(K_{\partial B_R}^k\right)^*[Y_m](x)=\lambda_{m,k,R} Y_m(\hat{x}),
\]
where
\[
\lambda_{m,k,R}=-\frac{1}{2R}\gamma_{m,k,R} j_m(kR)-\frac{ikR}{4}E_{m,k,R}.
\]
Furthermore, one can obtain that the eigenvalues for the single layer potential operator on the boundary $\partial B_R$
\begin{equation}\label{eq:eigenvalue_single_inside}
  S_{\partial B_R}^{k}[Y_m](x)=\chi_{m,k,R} Y_m(\hat{x}) \quad x\in \partial B_R,
\end{equation}
where
\[
 \chi_{m,k,R}=\gamma_{m,k,R} j_m(kR),
\]
with $\gamma_{m,k,R}$ given in \eqref{eq:coeff_gamma}.

If using the expression of $S_{\partial B_R}^{k}[Y_m](x)$ outside $B_R$, by the same argument one can obtain that
\[
 \left(K_{\partial B_R}^k\right)^*[Y_m](x)=\lambda_{m,k,R} Y_m(\hat{x}),
\]
where
\[
 \lambda_{m,k,R}=-\frac{1}{2R}\alpha_{m,k,R} h_m^{(1)}(kR)-\frac{ikR}{4}E_{m,k,R},
\]
with
\[
 \alpha_{m,k,R}=\frac{(2-ikRE_m)R}{2\left(2kR h_n^{(1)\prime}(kR)+ h_n^{(1)}(kR)\right)}.
\]
Analogously, the eigenvalues for the single layer potential operator on the boundary $\partial B_R$ is
\begin{equation}\label{eq:eigenvalue_single_outside}
  S_{\partial B_R}^{k}[Y_m](x)=\chi_{m,k,R} Y_m(\hat{x}) \quad x\in \partial B_R,
\end{equation}
where
\[
 \chi_{m,k,R}=\alpha_{m,k,R} h_m^{(1)}(kR).
\]
The eigenvalues in \eqref{eq:eigenvalue_single_inside} and \eqref{eq:eigenvalue_single_outside} should be the same, thus one has that
\begin{equation}\label{eq:identity_eigenvalue}
  \gamma_{m,k,R} j_m(kR)=\alpha_{m,k,R} h_m^{(1)}(kR).
\end{equation}
Substituting the expressions of $\gamma_{m,k,R}$ and $\alpha_{m,k,R}$ into the equation \eqref{eq:identity_eigenvalue} yields that the complex integral expression of $E_{m,k,R}$ in \eqref{eq:expression_E_m} can be written as
\begin{equation}
  E_{m,k,R} = 2kR\left( h_m^{(1)}(kR)j_m^{\prime}(kR) + j_m(kR)h_m^{(1){\prime}}(kR) \right) + 2 j_m(kR)h_m^{(1)}(kR) ,
\end{equation}
which follows from
\[
  j_n(t)h_n^{(1)\prime}(t)-j_n^{\prime}(t)h_n^{(1)}(t)=\frac{i}{t^2}.
\]
Finally, one can show that the eigenvalue for the single layer operator $S_{\partial B_R}^{k}$ can be simplified as
\[
  \chi_{m,k,R}=-i k R^2 h_m^{(1)}(kR) j_m(kR) ,
\]
and the eigenvalue for the operator $\left(K_{\partial B_R}^k\right)^*$ is
\[
 \lambda_{m,k,R}=\frac{1}{2}-i k^2 R^2 j_n^{\prime}(kR) h_n^{(1)}(kR)=-\frac{1}{2}-i k^2 R^2 j_n(kR) h_n^{(1)\prime}(kR) .
\]

This proof is complete.
\end{proof}

\begin{rem}\label{rem:lll1}
We believe the condition \eqref{eq:assumption_kR} should always hold. However, the verification of it is fraught with difficulties, and we include it as a condition. It is also remarked that from \eqref{eq:eigenvalue_lambda_m}, one can verify that the eigenvalues are complex numbers for $k\neq 0$.
\end{rem}

\subsection{Resonance result}
After achieving the spectral system of the Neumann-Poinc\'are operator, we can apply it to derive the resonance result for the Helmholtz system \eqref{eq:general equation_without_core}. Consider \eqref{eq:general equation_without_core} and the solution could be represented as the following integral ansatz, 
\begin{equation}
  u(x)=\left\{
         \begin{array}{ll}
           S_{\partial B_{r_e}}^{k_{1,\delta}}[\varphi](x), & x\in B_{r_e}, \\
           S_{\partial B_{r_e}}^{k}[\psi](x) + F(x), & x\in \mathbb{R}^3\backslash \overline{B_{r_e}},
         \end{array}
       \right.
\end{equation}
with $(\varphi,\psi)\in L^2(\partial B_{r_e})\times L^2(\partial B_{r_e})$ and, $k_{1,\delta}$ and $F(x)$ given in \eqref{eq:potential_F} and \eqref{eq:definition_k_1} respectively.
By using the transmission conditions across $\partial B_{r_e}$, one then has
\begin{equation}\label{eq:solution_N-P}
  \left\{
    \begin{array}{ll}
      S_{\partial B_{r_e}}^{k_{1,\delta}}[\varphi](x) = S_{\partial B_{r_e}}^{k}[\psi](x) + F(x), \\
      (\epsilon_s+i\delta)\frac{\partial}{\partial\nu}\left(S_{\partial B_{r_e}}^{k_{1,\delta}}[\varphi](x)\right) = \frac{\partial}{\partial\nu}\left(S_{\partial B_{r_e}}^{k}[\psi](x) + F(x)\right).
    \end{array}
  \right.
\end{equation}
Substituting the jump relation \eqref{eq:jump_relation} into the last equation \eqref{eq:solution_N-P} yields that
\begin{equation}\label{eq:equation_N-P}
  \left[
    \begin{array}{cc}
      S_{\partial B_{r_e}}^{k_{1,\delta}} & -S_{\partial B_{r_e}}^{k} \\
      \left(\epsilon_s+i\delta\right)\left(-\frac{1}{2}I + \left(K_{\partial B_R}^{k_{1,\delta}}\right)^*\right) & -\left(\frac{1}{2}I+ \left(K_{\partial B_R}^{k}\right)^*\right) \\
    \end{array}
  \right]
\left[
  \begin{array}{c}
    \varphi \\
    \psi \\
  \end{array}
\right]=
\left[
  \begin{array}{c}
    F \\
    \frac{\partial}{\partial{\nu}}F \\
  \end{array}
\right].
\end{equation}
Since the potential $F(x)$ for $x\in\overline{B_{r_e}}$ can be represented as in \eqref{eq:potential_F_expression}, then for $x\in\partial B_{r_e}$, $F(x)$ and $\frac{\partial}{\partial\nu}F(x)$ can be written as
\begin{equation}\label{eq:potential_F_N-P}
 F(x)=\sum_{n=0}^{\infty} \beta_n j_n(kr_e)Y_n(\hat{x}), \quad \frac{\partial}{\partial\nu}F(x)=\sum_{n=0}^{\infty} k\beta_n j_n^{\prime}(kr_e)Y_n(\hat{x}).
\end{equation}
Substituting \eqref{eq:potential_F_N-P} into the equation \eqref{eq:equation_N-P}, one can obtain
\begin{equation}
  \varphi(x)=\sum_{n=0}^{\infty}\hat{\varphi}_n Y_n(\hat{x}),
\end{equation}
where
\begin{equation}\label{eq:expression_varphi_n}
  \hat{\varphi}_n=\frac{k\chi_{n,k,r_e}j_n^{\prime}(kr_e)-(1/2+\lambda_{n,k,r_e})j_n(kr_e)} {(\epsilon_s+i\delta)(-1/2+\lambda_{n,k_1,r_e})\chi_{n,k,r_e}-(1/2+\lambda_{n,k,r_e})\chi_{n,k_1,r_e}}\beta_n.
\end{equation}
Therefore from Green's formula, one has that
\begin{equation*}
  \begin{split}
    \mathscr{E}_{\delta}[u] & = \delta(k_{1,\delta})^2\int_{B_{r_e}}|u|^2 dx + \delta\int_{\partial B_{r_e}}\frac{\partial u}{\partial\nu}\overline{u} ds(x)\\
      & = \delta(k_{1,\delta})^2\int_{B_{r_e}}|S_{\partial B_{r_e}}^{k_{1,\delta}}[\varphi]|^2 dx + \delta\int_{\partial B_{r_e}}\left(-\frac{1}{2} + \left(K_{\partial B_R}^{k_{1,\delta}}\right)^* \right)[\varphi]\overline{S_{\partial B_{r_e}}^{k_{1,\delta}}[\varphi]} ds(x) \\
      & =\sum_{n=0}^{\infty}\delta |\hat{\varphi}_n|^2 \left((k_{1,\delta})^2 |\gamma_{n,k_{1,\delta},r_e}|^2\int_{0}^{r_e}|j_n(k_{1,\delta}r)r|^2dr +(-\frac{1}{2}+\lambda_{n,k_{1,\delta},r_e})\overline{\chi_{n,k_{1,\delta},r_e}}r_e^2   \right).
  \end{split}
\end{equation*}
Thus if we can determine $\epsilon_s+i\delta$ such that for some $n_0\in\mathbb{N}$,
\begin{equation}\label{eq:spectrum_resonance}
  \delta |\hat{\varphi}_{n_0}|^2\rightarrow\infty\quad\mbox{as}\ \ \delta\rightarrow\delta_0, 
\end{equation}
then the resonance occurs. Substituting the expressions of $\lambda_{n,k,r_e}$, $\chi_{n,k_1,r_e}$, $\lambda_{n,k_1,r_e}$ and $\chi_{n,k,r_e}$ given in Theorem \ref{thm:eigenvalue_k-star} into \eqref{eq:expression_varphi_n}, the condition \eqref{eq:spectrum_resonance} can be shown to be reduced to \eqref{eq:fourier_resonance}. Therefore one can derive the same results in Theorem~\ref{thm:main1}. By using the spectral properties of the Nemann-Poinc\'are operator, one can also prove Theorem~\ref{thm:CALR} in a similar manner.

\section{Plasmon resonance and ALR results in two dimensions}

In this section, we extend our 3D resonance and ALR results in Section~2 to the 2D case. The study is pretty much parallel to the 3D case and in what follows, we shall present the main results and then sketch their proofs. 

\subsection{Resonance result with no core in $\mathbb{R}^2$}

We first consider the Helmholtz system \eqref{eq:general equation_without_core} in the two-dimensional case. 
%
Let $J_n(t)$ and $H_n^{(1)}(t)$ be, respectively, the Bessel and Hankel functions of order $n\in\mathbb{N}$. The fundamental solution in \eqref{eq:fundamental_solution} in the 2D case is replaced by $G^{k}(x)=-\frac{i}{4} H_0^{(1)}(k|x|)$. The Newtonian potential  the potential $F(x)$ defined in \eqref{eq:potential_F} has the following series representation for $x\in B_{R_1}$,
\begin{equation}\label{eq:potential_F_expression_two}
 F(x)=\sum_{n=-\infty}^{\infty} \beta_n J_n(kr) e^{in\theta},
\end{equation}
where and also in what follows, we assume that $J_{-n}(t)=J_{n}(t)$. Following a similar argument to that in \eqref{eq:solution_without_core_ge}--\eqref{eq:en1}, one can show that 
\begin{equation}\label{eq:solution_without_core_ge_two}
  u(x)=\left\{
         \begin{array}{ll}
          {\sum_{n=-\infty}^{\infty} a_n J_n(k_{1,\delta}r) e^{in\theta}}, & x\in B_{r_e}, \\
          {\sum_{n=-\infty}^{\infty} b_n J_n(kr) e^{in\theta} + c_n H_n^{(1)}(kr) e^{in\theta}} , & x\in B_{R_1}\backslash \overline{B_{r_e}}.
         \end{array}
       \right.
\end{equation}
with
\begin{equation}\label{eq:solution_without_core_two}
  \left\{
    \begin{array}{ll}
  {  a_n=b_n \frac{J_n^{\prime}(kr_e) H_n^{(1)}(kr_e)- H_n^{(1)\prime}(kr_e)J_n(kr_e)} {\sqrt{\epsilon_s+i\delta} J_n^{\prime}(k_{1,\delta} r_e)H_n^{(1)}(kr_e) - H_n^{(1)\prime}(kr_e)J_n(k_{1,\delta} r_e)  }},\medskip\\
  {  c_n=b_n \frac{J_n^{\prime}(kr_e) J_n(k_{1,\delta} r_e) -\sqrt{\epsilon_s+i\delta}J_n^{\prime}(k_{1,\delta} r_e)J_n(kr_e)}{\sqrt{\epsilon_s+i\delta} J_n^{\prime}(k_{1,\delta} r_e)H_n^{(1)}(kr_e) - H_n^{(1)\prime}(kr_e)J_n(k_{1,\delta} r_e)  },}
    \end{array}
  \right.
\end{equation}
and $b_n=\beta_n$. Moreover, we have that
\begin{equation}
 \begin{split}
   \mathscr{E}_{\delta}[u] & = \delta \left( (k_{1,\delta})^2\int_{B_{r_e}}|u|^2 dx + \int_{\partial B_{r_e}}\frac{\partial u}{\partial\nu}\overline{u} ds(x)\right) \\
     & =\sum_{n=-\infty}^{\infty}2\pi\delta |a_n|^2 \left( (k_{1,\delta})^2\int_{0}^{r_e}|J_n(k_{1,\delta}r)|^2 r dx + k_{1,\delta}r_eJ_n^{\prime}(k_{1,\delta}r_e)\overline{J_n(k_{1,\delta}r_e)} \right),
 \end{split}
\end{equation}
which shows that if there exists $n_0\in\mathbb{N}$ such that
\begin{equation}\label{eq:condition_resonance_two}
 \delta |a_{n_0}|^2\rightarrow \infty\quad\mbox{as}\ \ \delta\rightarrow\delta_0,
\end{equation}
then the plasmon resonance occurs. Next, by using the fact that
 \[
  J_n(t)H_n^{(1)\prime}(t)-J_n^{\prime}(t)H_n^{(1)}(t)=\frac{2i}{\pi t},
\]
one has that
\[
   a_n=\frac{1}{\pi kr_e} \frac{-2i \beta_n} {\sqrt{\epsilon_s+i\delta} J_n^{\prime}(k_{1,\delta} r_e)H_n^{(1)}(kr_e) - H_n^{(1)\prime}(kr_e)J_n(k_{1,\delta} r_e)  }.
\]
Hence, if there exists plasmon parameter $\epsilon_s+i\delta$ such that
\begin{equation}\label{eq:fourier_resonance_two}
  \sqrt{\epsilon_s+i\delta} J_n^{\prime}(k_{1,\delta} r_e)H_n^{(1)}(kr_e) - H_n^{(1)\prime}(kr_e)J_n(k_{1,\delta} r_e)\rightarrow 0,
\end{equation}
as $\delta\rightarrow \delta_0$, then the condition \eqref{eq:condition_resonance_two} is fulfilled. 

\subsubsection{Plasmon resonance in two dimensions}

Similar to our study in the three dimensional case, we first present some simple numerical constructions.

Set
\[
 k=1,\quad \mbox{and} \quad r_e=1,
\]
and for simplicity, we assume that the coefficients of the potential $F(x)$ in \eqref{eq:potential_F_expression_two} are given as
\begin{equation}
  \beta_n=\left\{
            \begin{array}{ll}
              1, & n=n_0, \\
              0, & n\neq n_0,
            \end{array}
          \right.
\end{equation}
The plasmon parameters are chosen as follows
\begin{equation}
  \epsilon_s=\epsilon_{n_0} \quad \mbox{and} \quad \delta=\delta_{n_0},
\end{equation}
which depends on $n_0$. With the help of numerical simulation, the following are the critical values of the resonant plasmon parameters
\begin{equation}\label{eq:data_two}
  \left\{
    \begin{array}{ll}
      \epsilon_{n_0}=-0.422878$ \quad $\delta_{n_0}=0.782261, & \mbox{for}\quad n_0=1, \\
      \epsilon_{n_0}=-0.674330$ \quad $\delta_{n_0}=0.115866, & \mbox{for}\quad n_0=2, \\
      \epsilon_{n_0}=-0.864384$ \quad $\delta_{n_0}=0.006257, & \mbox{for}\quad n_0=3, \\
      \epsilon_{n_0}=-0.931486$ \quad $\delta_{n_0}=0.000143, & \mbox{for}\quad n_0=4.
    \end{array}
  \right.
\end{equation}
Fig.~5 plots the change of the energy $\mathscr{E}_\delta$ against the change of the loss parameter $\delta$ when $\epsilon_{n_0}$ is fixed for $n_0=1,2,3,4$, which show the same behaviours as the three dimensional case.

\begin{figure}
  \centering
 {\includegraphics[width=4cm]{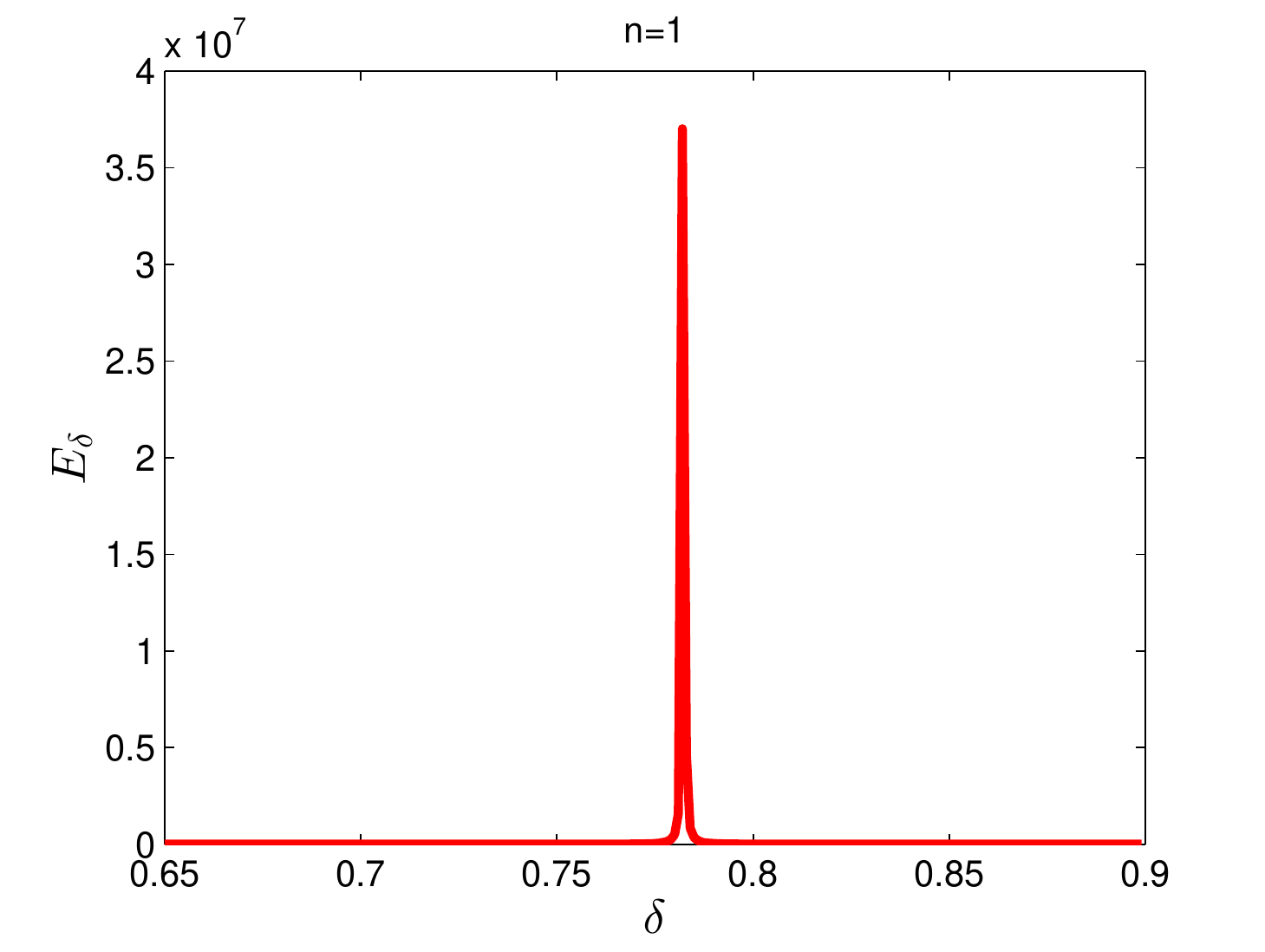}}
 {\includegraphics[width=4cm]{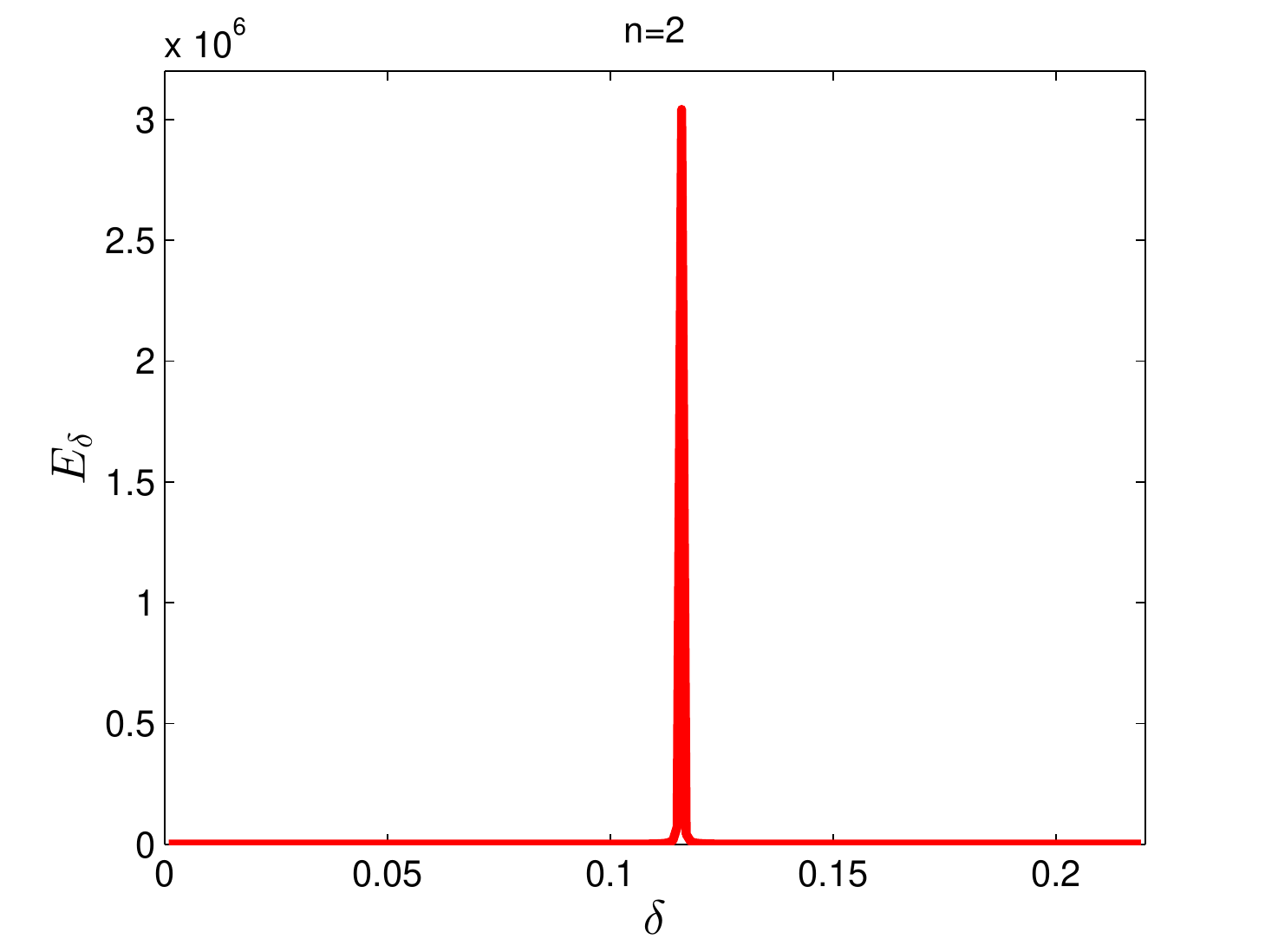}\\}
 {\includegraphics[width=4cm]{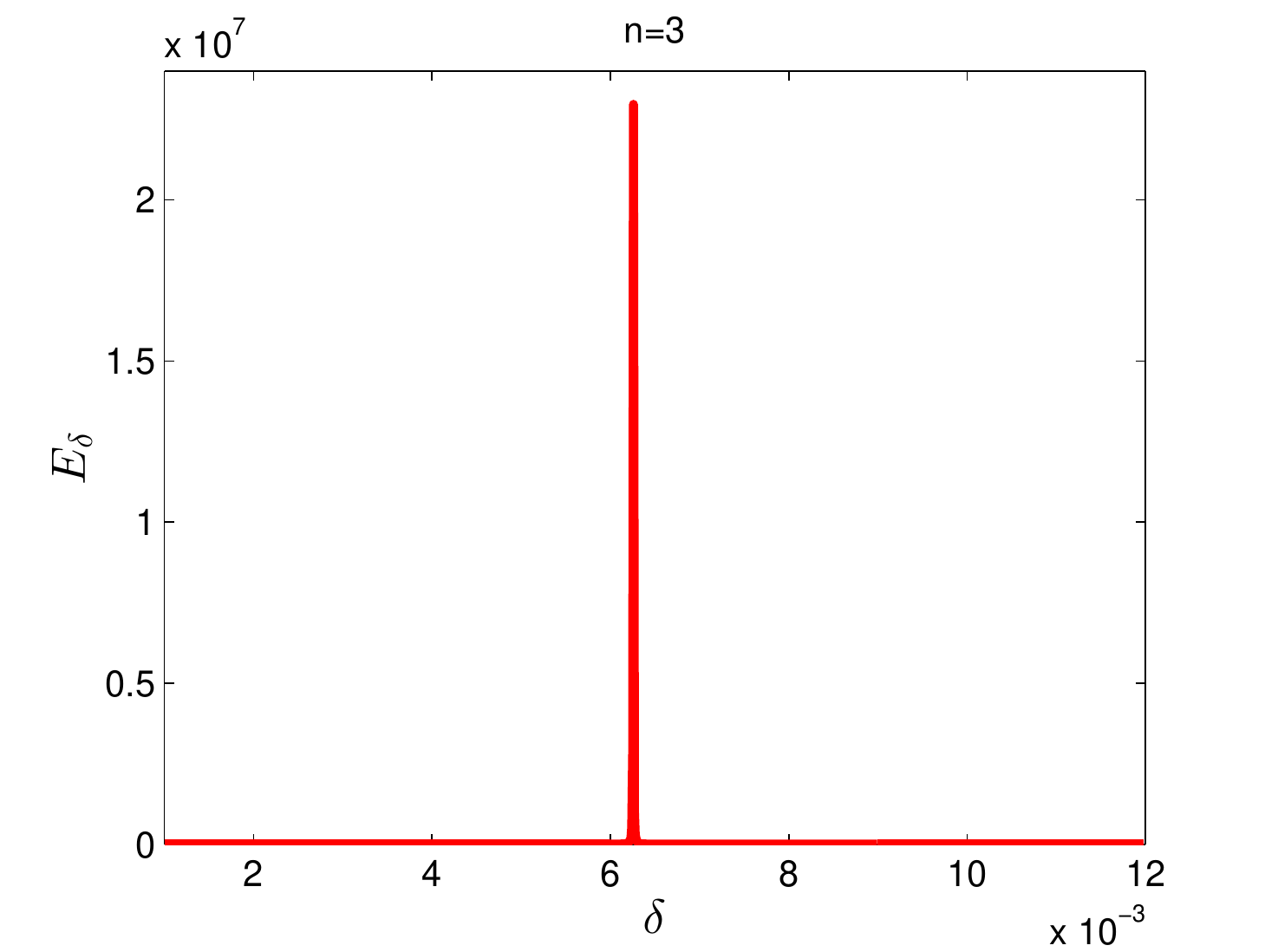}}
 {\includegraphics[width=4cm]{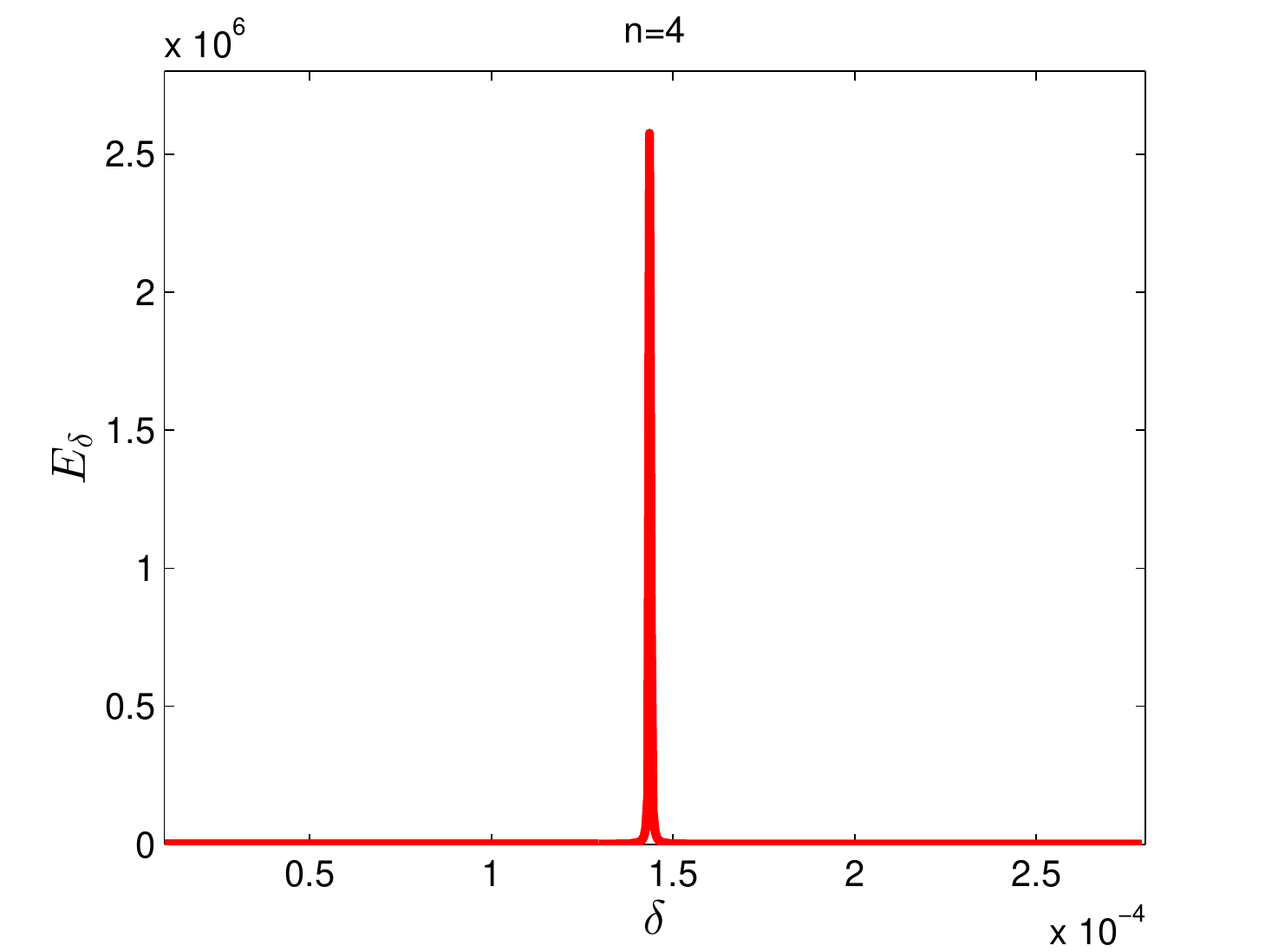}}
  \caption{Resonance}
\end{figure}

Next, we present the construction of a general 2D plasmonic structure that can induce resonance. To that end, we recall the following asymptotic properties of the functions $J_n(t)$ and $H_n^{(1)}(t)$ for sufficiently large $n\in\mathbb{N}$,
\begin{equation}\label{eq:asymptotic_J}
  J_n(t)=\hat{J}_n(t)\left(1+ \check{J}_n(t) \right),\quad  H_n^{(1)}(t)=\hat{H}_n^{(1)}(t) \left(1+\check{H}_n^{(1)}(t) \right),
\end{equation}
where  $\hat{J}_n(t)$, $\check{J}_n(t)$, $\hat{H}_n^{(1)}(t)$ and $\check{H}_n^{(1)}(t)$ have the following properties,
\begin{equation}\label{eq:dd2}
\hat{J}_n(t)=\frac{t^n}{2^n n!}, \quad \check{J}_n(t)=\mathcal{O}\left(\frac{1}{n}\right),
\hat{H}_n^{(1)}(t)=\frac{2^n (n-1)!}{\pi i t^n},  \quad \check{H}_n^{(1)}(t)=\mathcal{O}\left(\frac{1}{n}\right). 
\end{equation}

\begin{thm}\label{thm:main2d1}
Consider the Helmholtz system \eqref{eq:general equation_without_core} in $\mathbb{R}^2$, where the Newtonian potential $F$ of the source $f$ is given in \eqref{eq:potential_F_expression_two}. Suppose that for $F$ in \eqref{eq:potential_F_expression_two}, there exists an index $n_0\in\mathbb{N}$ such that $\beta_{n_0}\neq 0$ and $n_0$ is sufficiently large such that when $n\geq n_0$, $J_n(t)$ and $H_n^{(1)}(t)$ enjoy the asymptotic expansions in \eqref{eq:asymptotic_J}. Let the plasmon parameters be chosen of the following form
\begin{equation}\label{eq:ee1}
\epsilon_s=-1, \quad \delta\in\mathbb{R}_+\quad\mbox{and}\quad \delta\ll 1. 
\end{equation}
Then if the plasmon configuration fulfils the following condition, 
\begin{equation}\label{eq:k_no_two}
\begin{split}
 \hat{J}_{n_0}(k_{1,\delta} r_e) \hat{H}_{n_0}^{(1)}(kr_e)  \left(  \sqrt{\epsilon_s+i\delta}  \check{J}_{n_0}(t)^{\prime}(k_{1,\delta} r_e)\left(1+\check{H}_{n_0}^{(1)}(kr_e)\right)  \right.& \\
               \left. -\check{H}_{n_0}^{(1)\prime}(kr_e)\left(1+\check{J}_{n_0}(k_{1,\delta} r_e)\right)   \right)&=0,   \\
\end{split}	
\end{equation}
one has $\mathscr{E}_\delta[u]\sim\delta^{-1}$. 
\end{thm}
\begin{proof}
By using \eqref{eq:ee1} and \eqref{eq:k_no_two}, together with straightforward calculations, one can show that 
\begin{equation}
  \left|\sqrt{\epsilon_s+i\delta} J_{n_0}^{\prime}(k_{1,\delta}r_e)H_{n_0}^{(1)}(kr_e) - H_{n_0}^{(1)\prime}(kr_e)J_{n_0}(k_{1,\delta}r_e) \right|\approx \delta \left(1+\mathcal{O}\left(\frac{1}{n_0}\right) \right).
\end{equation}
From the solution given in \eqref{eq:solution_without_core_two} and with the help of Green's formula, one has that
\begin{equation}
 \begin{split}
   E_{\delta} & =\delta\int_{B_{r_e}}|\nabla u|^2dx = \delta k_{1,\delta}^2\int_{B_{r_e}}|u|^2dx + \delta\int_{\partial B_{r_e}} \frac{\partial u}{\partial \nu} \overline{u}ds(x)\\
     & \geq \delta k_{1,\delta}^2\int_{B_{r_e}}|a_{n_0} J_{n_0}(k_{1,\delta}r) e^{in_0\theta}|^2dx \\
     & \quad +\delta\int_{\partial B_{r_e}} a_{n_0} k_{1,\delta} J_{n_0}^{\prime}(k_{1,\delta}r_e)  \overline{ \left(a_{n_0}   J_{n_0}(k_{1,\delta}r_e) \right)} |e^{in_0\theta}|^2 ds(x) \\
     & \approx \frac{|\beta_{n_0}|^2}{\delta}\left(1+\mathcal{O}\left(\frac{1}{n_0}\right) \right),
 \end{split}
\end{equation}
which completes the proof by noting that $\beta_{n_0}\neq 0$. 
\end{proof}

\begin{rem}
Similar to Remark 2.1, we can numerically verify that the equation in \eqref{eq:k_no_two} yields a nonempty set of parameters. We set $r_e=1, n_0=300, \epsilon_s=-1$ and $\delta=0.5^{n_0}$, and let $k$ be a free parameter. Fig.~6 plots the quantity in the LHS of \eqref{eq:k_no_two} against $k$. One readily sees that there do exist $k$'s such that \eqref{eq:k_no_two} holds. Hence, resonance occurs with the aforesaid parameters at those $k$'s. One can also fix $k$ and determine the other parameters by solving \eqref{eq:k_no_two}. 
\end{rem}

\begin{figure}[t]
  \centering
 {\includegraphics[width=5cm]{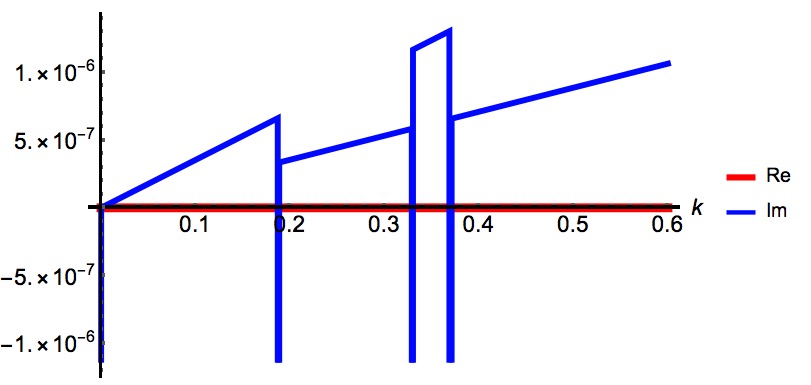}}
  \caption{The real and imaginary parts of the LHS quantity in \eqref{eq:k_no_two} with respect the change of the wavenumber $k$. }
  \label{fig:3}
\end{figure}

\subsection{ALR result with a core in $\mathbb{R}^2$}

In this subsection, we consider the Helmholtz system \eqref{eq:helm1}--\eqref{eq:mc1} in $\mathbb{R}^2$ with $r_i\neq 0$ , and show that ALR can be induced. Let the Newtonian potential of the source $f$ is again given in \eqref{eq:potential_F_expression_two}. By following a similar argument to that in deriving \eqref{eq:e1}--\eqref{eq:2.29}, one can show that the solution to \eqref{eq:helm1}--\eqref{eq:mc1} in $B_{R_1}$ can be expressed as
\begin{equation}
  u(x)=\left\{
         \begin{array}{ll}
          { \sum_{n=-\infty}^{\infty} a_n J_n(\frac{kr}{\sqrt{\epsilon_c}}) e^{in\theta}, }& x\in B_{r_i}, \\
         { \sum_{n=-\infty}^{\infty} b_n J_n(k_{1,\delta}r) e^{in\theta} + c_n H_n^{(1)}(k_{1,\delta}r) e^{in\theta} ,} & x\in B_{r_e}\backslash B_{r_i}  \\
          {\sum_{n=-\infty}^{\infty} e_n J_n(kr) e^{in\theta} + d_n H_n^{(1)}(kr) e^{in\theta} ,} & x\in B_{R_1}\backslash B_{r_e},
         \end{array}
       \right.
\end{equation}
with $a_n={\tilde{a}_n}/{g_n }, b_n={\tilde{b}_n}/{g_n }, c_n={\tilde{c}_n}/{g_n }, d_n={\tilde{d}_n}/{g_n }$, where
\begin{align}
\tilde{a}_n= & e_n \sqrt{\epsilon_s+i\delta}\times \zeta_n \times \left(J_n^{\prime}(k_{1,\delta}r_i) H_n^{(1)}(k_{1,\delta}r_i) - H_n^{(1)\prime}(k_{1,\delta}r_i)J_n(k_{1,\delta}r_i)\right),\\
\tilde{b}_n= & e_n \times \zeta_n\times \left(\sqrt{\epsilon_c}J_n^{\prime}(\frac{kr_i}{\sqrt{\epsilon_c}}) H_n^{(1)}(k_{1,\delta}r_i) - \sqrt{\epsilon_s+i\delta} H_n^{(1)\prime}(k_{1,\delta}r_i)J_n(\frac{kr_i}{\sqrt{\epsilon_c}})\right),  \\
  \tilde{c}_n= & e_n \times \zeta_n\times \left(\sqrt{\epsilon_s+i\delta} J_n^{\prime}(k_{1,\delta}r_i)J_n(\frac{kr_i}{\sqrt{\epsilon_c}}) - \sqrt{\epsilon_c}J_n^{\prime}(\frac{kr_i}{\sqrt{\epsilon_c}}) J_n(k_{1,\delta}r_i) \right),\\
   \zeta_n:=& J_n^{\prime}(kr_e) H_n^{(1)}(kr_e) - H_n^{(1)\prime}(kr_e)J_n(kr_e),
\end{align}
and
\begin{equation}
  \begin{split}
   & \tilde{d}_n= e_n \left( \sqrt{\epsilon_c} \left(H_n^{(1)}(k_{1,\delta}r_i)J_n(k_{1,\delta}r_e)- H_n^{(1)}(k_{1,\delta}r_e)J_n(k_{1,\delta}r_i)\right) j_n^{\prime}(\frac{kr_i}{\sqrt{\epsilon_c}})j_n^{\prime}(kr_e) + \right. \\
    & \tau_{s,\delta}\left( H_n^{(1)\prime}(k_{1,\delta}r_i)J_n^{\prime}(k_{1,\delta}r_e) -H_n^{(1)\prime}(k_{1,\delta}r_e)J_n^{\prime}(k_{1,\delta}r_i) \right) J_n(\frac{kr_i}{\sqrt{\epsilon_c}})J_n(kr_e) + \\
    & \sqrt{\tau_{s,\delta}}\left( J_n^{\prime}(k_{1,\delta}r_i)H_n^{(1)}(k_{1,\delta}r_e) - H_n^{(1)\prime}(k_{1,\delta}r_i)J_n(k_{1,\delta}r_e) \right) J_n(\frac{kr_i}{\sqrt{\epsilon_c}})J_n^{\prime}(kr_e) + \\
    & \left. \sqrt{\tau_{s,\delta}}\sqrt{\epsilon_c} \left( H_n^{(1)\prime}(k_{1,\delta}r_e)J_n(k_{1,\delta}r_i) - J_n^{\prime}(k_{1,\delta}r_e)H_n^{(1)}(k_{1,\delta}r_i)\right) J_n^{\prime}(\frac{kr_i}{\sqrt{\epsilon_c}})J_n(kr_e) \right),
  \end{split}
\end{equation}
and
\begin{equation}
  \begin{split}
   & g_n = \tau_{s,\delta}\left( H_n^{(1)\prime}(k_{1,\delta}r_e)J_n^{\prime}(k_{1,\delta}r_i) -H_n^{(1)\prime}(k_{1,\delta}r_i)J_n^{\prime}(k_{1,\delta}r_e) \right)H_n^{(1)}(kr_e) J_n(\frac{kr_i}{\sqrt{\epsilon_c}}) +  \\
      & \sqrt{\tau_{s,\delta}}\left(H_n^{(1)\prime}(k_{1,\delta}r_i)J_n(k_{1,\delta}r_e)  -J_n^{\prime}(k_{1,\delta}r_i) H_n^{(1)}(k_{1,\delta}r_e) \right)H_n^{(1)\prime}(kr_e) J_n(\frac{kr_i}{\sqrt{\epsilon_c}}) + \\
      & \sqrt{\epsilon_c}\left( H_n^{(1)}(k_{1,\delta}r_e)J_n(k_{1,\delta}r_i) -H_n^{(1)}(k_{1,\delta}r_i)J_n(k_{1,\delta}r_e) \right)J_n^{\prime}(\frac{kr_i}{\sqrt{\epsilon_c}})H_n^{(1)\prime}(kr_e) + \\
      &  \sqrt{\epsilon_c}\sqrt{\tau_{s,\delta}}\left( J_n^{\prime}(k_{1,\delta}r_e) H_n^{(1)}(k_{1,\delta}r_i) - H_n^{(1)\prime}(k_{1,\delta}r_e)J_n(k_{1,\delta}r_i)    \right) J_n^{\prime}(\frac{kr_i}{\sqrt{\epsilon_c}})H_n^{(1)}(kr_e),
  \end{split}
\end{equation}
with $\tau_{s,\delta}:=\epsilon_s+\delta$ and $k_{1,\delta}$ given in \eqref{eq:definition_k_1}. One also has that $e_n=\beta_n$. Similar to \eqref{eq:potential_F_constain}, we impose the constraint on $f$ such that the corresponding Newtonian potential $F(x)$ is of the following form
\begin{equation}\label{eq:potential_F_constain_two}
 F(x)=\sum_{|n|=N}^{\infty} \beta_n J_n(kr)e^{in\theta},
\end{equation}
for some sufficiently large $N\in\mathbb{N}$ such that when $n\geq N$, $J_n(r)$ and $H_n^{(1)}(r)$ enjoy the asymptotic properties given in \eqref{eq:asymptotic_J}. For the subsequent use, we also define the following functions:
\begin{equation*}
\begin{split}
&\widetilde{\varphi}_1(n,b_1,b_2,r_1,r_2):=\hat{J}_n(r_1) \hat{H}_n^{(1)}(r_2)  \left( b_1 \check{J}_n^{\prime}(r_1)\left(1+\check{H}_n^{(1)}(r_2)\right)   -b_2\check{H}_n^{(1)\prime}(r_2)\left(1+\check{J}_n(r_1)\right) \right),\\
&\widetilde{\varphi}_2(n,b_1,b_2,r_1,r_2):= b_1 J_n^{\prime}(r_1) H_n^{(1)}(r_2) - b_2 H_n^{(1)\prime}(r_2) J_n(r_1) , 
\end{split}
\end{equation*}
where  $\hat{J}_n(t)$, $\check{J}_n(t)$, $\hat{H}_n^{(1)}(t)$ and $\check{H}_n^{(1)}(t)$ are given \eqref{eq:asymptotic_J} and \eqref{eq:dd2}.

\begin{thm}\label{thm:CALR_two}
Consider the Helmholtz system \eqref{eq:helm1}--\eqref{eq:mc1} with $r_i\neq 0$ in $\mathbb{R}^2$ with the Newtonian potential $F$ of the source $f$ satisfying \eqref{eq:potential_F_constain_two}. Let the plasmon configuration be chosen of the following form
 \begin{equation}\label{eq:CALR2d1}
   \epsilon_s=-1, \quad  \epsilon_c=1\quad\mbox{and}\quad \delta=\rho^{n_0},
 \end{equation}
 where $\rho:=r_i/r_e<1$ and $n_0\in \mathbb{N}$ with $n_0\gg 1$. If the plasmon configuration fulfils the following condition,
\begin{equation}\label{eq:k_two}
\begin{split}
&\widetilde{\varphi}_1(n_0,\sqrt{\epsilon_c},\tau_{s,\delta},kr_i,k_{1,\delta} r_i)\widetilde{\varphi}_2(n_0,\tau_{s,\delta},1,k_{1,\delta} r_e,kr_e) + \widetilde{\varphi}_1(n_0,\tau_{s,\delta},1,k_{1,\delta} r_e,kr_e)\\
&\times\left(\widetilde{\varphi}_2(n_0,\sqrt{\epsilon_c},\tau_{s,\delta},k r_i,k_{1,\delta} r_i) - \widetilde{\varphi}_1(n_0,\sqrt{\epsilon_c},\tau_{s,\delta},kr_i,k_{1,\delta} r_i)   \right)=0,
\end{split}
\end{equation}
where $\tau_{s,\delta}=\sqrt{\epsilon_s+i\delta}$, then there is a critical radius $r_{*}:=\sqrt{r_e^3/r_i}$ such that if $f$ lies within this radius, $\mathscr{E}_{\delta}[u]\geq \mu_0^{n_0} n_0$ with $\mu_0>1$, and $u(x)$ remains bounded for $|x|>r_e^2/r_i$; and if $f$ lies outside this radius, $\mathscr{E}_{\delta}[u]$ is bounded by a constant depending only on $f, k$ and $r_e$. 
\end{thm}

\begin{proof}

The proof is similar to that of Theorem~\ref{thm:CALR} in the three dimensional case and in what follows, we mainly point out several major ingredients. Set $\tilde{\beta}_n:=\frac{{\beta}_n}{ 2^n n!}$. One can derive the following estimates when $n\geq N$, 
\begin{equation}\label{eq:coefficient_g_n_two}
  g_{n}\approx \delta^2 + \rho^{2n},\ \  \tilde{b}_{n}\approx i\delta \beta_{n},\ \ 
 \tilde{c}_{n}\approx \frac{n (kr_i)^{2n}}{4^{n}(n!)^2} \beta_{n},\ \  \tilde{d}_{n}\approx -\frac{\delta n (kr_e)^{2n}}{4^n (n!)^2}\beta_{n}.
\end{equation}
For $n_0$ sufficiently large, one can directly estimate with the help of \eqref{eq:coefficient_g_n_two} that 
\begin{equation}\label{eq:energy_estimate_with_core_two}
   \mathscr{E}_{\delta}\approx \sum_{|n|\geq N} \frac{n\beta_n^2 (k r_e)^{2n}}{(2^n n!)^2} \frac{\delta}{\delta^2+\rho^{2n}}\geq n_0\tilde{\beta}_{n_0}^2 \left(\frac{k^2 r_e^3}{r_i} \right)^{n_0}.
\end{equation}
If the source $f$ is supported inside the critical radius $r_*$, by a similar reasoning to \eqref{eq:lll2}, one can show that there exists $\tau\in\mathbb{R}_+$ such that
\begin{equation}\label{eq:lll3}
  \limsup_{n\rightarrow\infty}\left(\tilde{\beta}_n \right)^{1/n}=\sqrt{\frac{r_i}{k^2r_e^3}+\tau}. 
\end{equation}
Combining \eqref{eq:energy_estimate_with_core_two} and \eqref{eq:lll3}, one can then show that $\mathscr{E}[u]\geq \mu_0^{n_0} n_0$ for some $\mu_0>1$ as stated in the theorem. The boundedness of $u(x)$ for $|x|>r_e^2/r_i$ can be shown as follows. By virtue of \eqref{eq:coefficient_g_n_two}, one has that
\begin{equation}\label{eq:estimate_d_n0_two}
   |d_{n}H_{n}^{(1)}(kr)| \leq \frac{n (kr_e)^{2n}}{4^n (n!)^2} \frac{\delta}{\delta^2+\rho^{2n}} |\beta_{n}| |H_{n}^{(1)}(kr)|\leq |\tilde{\beta}_{n}| (kr_e)^n \left(\frac{r_e^2}{r_i}\right)^n \frac{1}{r^n},
\end{equation}
which in turn implies that when $|x|>r_e^2/r_i$,
\[
 |u(x)-F(x)|\leq \sum_{|n|\geq N}|\tilde{\beta}_{n}| (kr_e)^n\leq C.
\]

The proof is complete. 
\end{proof}

\begin{rem}
Similar to Remark 2.2, we can numerically verify that the equation in \eqref{eq:k_two} yields a nonempty set of parameters. For illustration, we set $n_0=300, r_i=0.5, r_e=1, \epsilon_c=1, \epsilon_s=-1$ and $\delta=(r_i/r_e)^{n_0}$, and let $k$ be a free parameter. Fig.~7 plots the quantity in the LHS of \eqref{eq:CALR2d1} against $k$ over an interval. It can be seen that there do exist $k$'s such that  \eqref{eq:CALR2d1} holds. 
\end{rem}

\begin{figure}[t]
  \centering
 {\includegraphics[width=5cm]{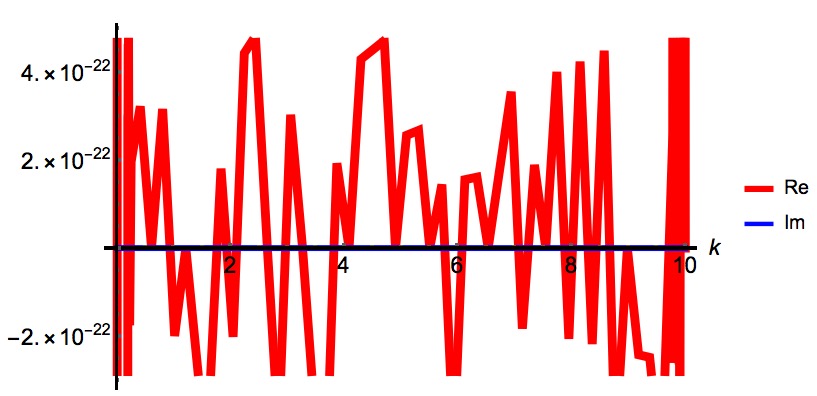}}
  \caption{The real part and imaginary part of left side of  the equation  \eqref{eq:k_two} with respect the change of the frequency $k$. }
  \label{fig:7}
\end{figure}
%
%

\section{Concluding remarks}

We have considered the plasmon resonance and cloaking due to anomalous localized resonance for the Helmholtz system in both two and three dimensions at finite frequencies beyond the quasistatic limits. Several novel constructions were presented which can induce the aforesaid phenomena. The major point of our study is that, if resonance occurs, then all the ingredients of a plasmon configuration form a nonlinear system. By picking the material parameters and external sources of certain specific forms, the above mentioned nonlinear system can be simplified to a verifiable one. The constraint on the sources is critical for the aforementioned simplification. Nevertheless, we would like to point out that it is unclear to us that the constraint is necessary. That is, it may happen that resonance can still be induced for more general sources which do not satisfy the constraint. Moreover, our study is mainly restricted to the radial geometry, and the analysis follows the classical Mie scattering theory method of using spherical wave expansions. Nevertheless, at a certain point, we also develop a spectral argument of using the spectral properties of the Neumann-Poincar\'e operator. The extension to the non-radial geometries is much challenging and is definitely worth future study.

\bigskip
%
%
%
%

\noindent{\bf Funding.}~~The work was supported by the FRG and startup grants from Hong Kong Baptist University, and Hong Kong RGC General Research Funds, 12302415 and 12302017.

\end{document}